\documentclass[a4paper,twoside,11pt]{article}
\usepackage[utf8]{inputenc}
\usepackage[english]{babel}
\usepackage[margin=2.5cm]{geometry} 

\usepackage{amssymb, amsfonts, amsthm, amsmath}
\usepackage[hidelinks]{hyperref}
\usepackage{caption} 
\usepackage{pgf,tikz} 
\usetikzlibrary{math,arrows,calc,through}

\usepackage[mathscr]{euscript}
\usepackage[percent]{overpic}
\usepackage{accents}
\usepackage{caption}
\usepackage{enumitem}
\setlist[enumerate, 1]{label=(\roman*)}

\usepackage{thmtools}
\declaretheorem[style=plain,name=Theorem,qed={\tiny$\blacksquare$},numberwithin=section]{theorem}
\declaretheorem[style=plain,name=Corollary,sibling=theorem,qed={\tiny$\blacksquare$}]{corollary}
\declaretheorem[style=definition,name=Definition,sibling=theorem,qed={\tiny$\blacksquare$}]{definition}
\declaretheorem[style=plain,name=Lemma,sibling=theorem,qed={\tiny$\blacksquare$}]{lemma}

\declaretheorem[style=definition,name=Remark,sibling=theorem,qed={\tiny$\blacksquare$}]{remark}

\numberwithin{equation}{section}

\renewcommand{\lor}{{\mathcal{\mathbf{L^3}}}}

\newcommand{\eucl}{{\mathcal{\mathbf{E^2}}}}
\newcommand{\unip}{\mathbf{U}_+} 

\renewcommand{\sp}{\mathcal{S}}
\newcommand{\osp}{\vec{\mathcal{S}}}

\newcommand{\oli}{\vec{\mathcal{L}}}
\newcommand{\ci}{\mathcal{C}}
\newcommand{\oci}{\vec{\mathcal{C}}}

\newcommand{\Zb}{\Z_\bullet} 
\newcommand{\Zw}{\Z_\circ} 

\newcommand{\cp}{p} 
\newcommand{\cpf}{p_{\fboxsep=-\fboxrule\sbox0{}\wd0=4pt\ht0=4pt\relax\fbox{\box0}}}
\newcommand{\cpm}{{\odot{p}}}
\newcommand{\cpb}{p_{\bullet}}
\newcommand{\cpw}{p_{\circ}}
\newcommand{\cpmb}{\odot p_{\bullet}}
\newcommand{\cpmw}{\odot p_{\circ}}

\newcommand{\Cp}{q} 
\newcommand{\Cpf}{q_{\fboxsep=-\fboxrule\sbox0{}\wd0=4pt\ht0=4pt\relax\fbox{\box0}}}
\newcommand{\Cpm}{{\odot{q}}}

\newcommand{\Cpmb}{\odot{q_{\bullet}}}
\newcommand{\Cpmw}{\odot q_{\circ}}

\newcommand{\iso}{h} 
\newcommand{\isof}{h_{\fboxsep=-\fboxrule\sbox0{}\wd0=4pt\ht0=4pt\relax\fbox{\box0}}}
\newcommand{\isom}{{\odot{h}}}
\newcommand{\isob}{h_{\bullet}}
\newcommand{\isow}{h_{\circ}}
\newcommand{\isomb}{\mathord{\odot}h_{\bullet}}
\newcommand{\isomw}{\mathord{\odot}h_{\circ}}
\newcommand{\isoc}{h_\boxplus}

\newcommand{\cc}{c}
\newcommand{\ccf}{c_{\fboxsep=-\fboxrule\sbox0{}\wd0=.8ex\ht0=.8ex\relax\fbox{\box0}}}
\newcommand{\ccm}{{\odot{c}}}
\newcommand{\ccb}{c_{\bullet}}
\newcommand{\ccw}{c_{\circ}}
\newcommand{\ccmb}{\odot c_{\bullet}}
\newcommand{\ccmw}{\odot c_{\circ}}
\newcommand{\ccc}{c_\boxplus}

\renewcommand{\d}{\mathrm d} 
\newcommand{\si}[1]{\texttt{(#1)}} 

\newcommand{\miq}[1]{\mathscr{M}(#1)}

\renewcommand{\S}{\mathbb{S}}
\newcommand{\Z}{\mathbb{Z}}
\newcommand{\R}{\mathbb{R}}
\newcommand{\C}{\mathbb{C}}

\newcommand{\RN}[1]{%
	\textup{\uppercase\expandafter{\romannumeral#1}}%
}

\usepackage{mathtools}
\mathtoolsset{showonlyrefs}

\setlist[enumerate]{itemsep=2pt, topsep=0pt}

\title{Discrete Lorentz surfaces and s-embeddings II:\\maximal surfaces}
\date{\today}

\author{
	Niklas Christoph Affolter\thanks{TU Berlin, Institute of Mathematics, Straße des 17.~Juni 136, 10623 Berlin, Germany.
	\textit{E-mail addresses}: \texttt{affolter~at~posteo.net, smeenk~at~math.tu-berlin.de}}\ \footnotemark[2] ,
  Felix Dellinger\thanks{TU Wien, Institut of Discrete Mathematics and Geometry, Wiedner Hauptstr. 8-10/104, A-1040 Vienna, Austria.
  \textit{E-mail addresses}: \texttt{felix.dellinger, christian.mueller, denis.polly~at~tuwien.ac.at}} ,\\
	Christian Müller\footnotemark[2] ,
	Denis Polly\footnotemark[2] ,
	Nina Smeenk\footnotemark[1] ,
}

\begin{document}

\maketitle

\begin{abstract}
	S-embeddings were introduced by Chelkak as a tool to study the conformal invariance of the thermodynamic limit of the Ising model. Moreover, Chelkak, Laslier and Russkikh introduced a lift of s-embeddings to Lorentz space, and showed that in the limit the lift converges to a maximal surface. They posed the question whether there are s-embeddings that lift to maximal surfaces already at the discrete level, before taking the limit. We answer this question in the positive. In a previous paper we identified a subclass of s-embeddings -- isothermic s-embeddings -- that lift to (discrete) S-isothermic surfaces, which were introduced by Bobenko and Pinkall as a discretization of isothermic surfaces. In this paper we identify a special class of isothermic s-embeddings that correspond to discrete S-maximal surfaces, translating an approach of Bobenko, Hoffmann and Springborn introduced for discrete S-minimal surfaces in Euclidean space. Additionally, each S-maximal surface comes with a 1-pa\-ra\-me\-ter family of associated surfaces that are isometric. This enables us to obtain an associated family of s-embeddings for each maximal s-embedding. We show that the Ising weights are constant in the associated family.
\end{abstract}

\newpage

\setcounter{tocdepth}{1}
\tableofcontents

\newpage

\section{Introduction}

This paper is the second in a series of two papers, the first being \cite{admpsiso}. The goal of the series is to relate recent developments in statistical mechanics to (not quite so recent) developments in discrete differential geometry. A brief introduction to the two topics can be found in \cite{admpsiso}.

Let us explain which developments we are referring to. An \emph{incircular net} is a map from a quad graph to the plane, such that each quad has an incircle. Chelkak \cite{chelkaksembeddings} used incircular nets to investigate properties of the (planar) \emph{Ising model}, calling them \emph{s-embeddings}. Moreover, he introduced a way to define the coupling constants of an Ising model given an incircular net. He then introduced a way to lift an incircular net to Lorentz space $\lor = \R^{2,1}$. Surprisingly, the properties of the Ising model are invariant under isometries of $\lor$. Moreover, he showed (under additional assumptions) that in the thermodynamic limit the lift converges to a maximal surface. Similar constructions and results were obtained by Chelkak, Laslier and Russkikh \cite{clrdimer,clrpembeddings} for the dimer model using \emph{conical nets}, also called \emph{t-embeddings} or \emph{Coulomb gauges} \cite{klrr} in the statistical mechanics community. However, in this paper we will focus on incircular nets. A question posed in \cite{clrpembeddings} was whether the lifts of some conical nets or incircular nets can be understood as maximal surfaces before taking the limit, as \emph{discrete maximal surfaces}. In this paper we show that the answer is: \emph{yes}. 

In the first paper \cite{admpsiso}, we investigated the geometric properties of the Lorentz lift introduced by Chelkak. The Lorentz lift of a quad of an incircular net is a quad in $\R^{2,1}$ such that each edge is isotropic (also called lightlike). We showed that
\begin{enumerate}
  \item in each quad there is a unique timelike Lorentz sphere that contains all four edges of the quad;
	\item at each vertex there is a null-sphere (a sphere of radius zero) that contains the isotropic lines of the four incident edges;
  \item at each edge the two incident timelike spheres are touching and the apices of the incident null-spheres lie on both incident timelike spheres.
\end{enumerate}
We call the resulting congruence of spheres a \emph{null congruence}, and we showed that every null congruence is the Lorentz lift of an incircular net. In particular, the incircles of the incircular net are the orthogonal projections of the smallest Euclidean circles (the contours) of the corresponding timelike spheres, the vertices of the incircular net are the projections of the apices of the null-spheres, and the edges of the incircular net are projections of the common isotropic line of the two null-spheres and the two timelike spheres.

Subsequently, we introduced a special class of null congruences which we call \emph{isothermic congruences}. They are characterized by the planarity of the four centers of the timelike spheres around each vertex. The corresponding \emph{isothermic incircular net} has the following property: if for every edge we take the other tangent to the two adjacent incircles, the collection of all these other tangents also constitutes an incircular net. Why do we call these congruences isothermic congruences? There is a well-known discretization of isothermic surfaces called \emph{S-isothermic nets} \cite{bpisoint,bhsminimal}. We showed that isothermic congruences are in 2:1 correspondence with S-isothermic nets, hence, there is a special case of incircular nets that does correspond to isothermic surfaces at the discrete level.

In the first part of this paper, we take the specialization sequence one step further. In particular, it is well known that (smooth) maximal surfaces -- that is, Riemannian surfaces with vanishing mean curvature in $\lor$ -- are isothermic surfaces. In \cite{bhsminimal}, the authors showed that there is a special class of S-isothermic nets which corresponds to discrete maximal surfaces (albeit in Euclidean space instead of Lorentz space). This is the advantage of incorporating the theory of incircular nets into discrete differential geometry: it enables us to take known results and apply them to incircular nets. More concretely, since there is a special class of S-isothermic nets that are discrete maximal surfaces, we identify a special class of isothermic congruences that we call \emph{maximal congruences}. These maximal congruences correspond to discrete maximal surfaces, and therefore we also obtain a special class of incircular isothermic nets that correspond to discrete maximal surfaces. Following \cite{bhsminimal}, discrete maximal surfaces -- and therefore maximal incircular nets -- are in bijection with hyperbolic orthogonal circle patterns, for which there exists a variational principle which allows one to obtain them uniquely from boundary data \cite{bsvariationalcp}.

In the second part of the paper, we study the associated family $f^\varphi$, $\varphi \in \S^1 = [0,2\pi]$ of a discrete maximal surface $f^0$. The associated family of a discrete maximal surfaces was also introduced in \cite{bhsconical}, as an analogue of the associated family of a maximal surface in the smooth setup. Note that each surface $f^\varphi$ in the associated family is also a maximal surface and is isometric to the initial maximal surface $f^0$. We show that the nets $f^\varphi$ in the associated family of a discrete maximal surface are in 2:1 correspondence with contact congruences. Furthermore, by slightly modifying the way the correspondence works, we are able to show that there is also a 2:1 correspondence of nets $f^\varphi$ with null congruences. Finally, we show that the $X$-variables in the associated family are independent of $\varphi$.
As a result of the correspondence between discrete maximal surfaces and maximal incircular nets, we are also able to construct an associated family of maximal incircular nets.

\subsection{Plan of the paper}

In Section~{\protect\ref{sec:smMaximal}} we recall some important properties of smooth maximal surfaces which motivate the discrete theory.
In Section~\ref{sec:isocc} and Section~\ref{sec:isocp} we recall the necessary theory of isothermic congruences and isothermic incircular nets respectively.
In Section~\ref{sec:koebe} we explain Koebe congruences, which correspond to the conformal data from which we construct maximal congruences. In Section~\ref{sec:christoffel} we recall how to obtain the Christoffel dual, which we use in Section~\ref{sec:maximal} to define maximal congruences from Koebe congruences. This process also allows for a Weierstraß representation as given in Section~\ref{sec:weierstrass}. In Section~\ref{sec:associated} and Section~\ref{sec:radii} we investigate the associated family of corresponding S-isothermic nets and null congruences of a maximal congruence. Finally, in Section~\ref{sec:xvariables} we show that every member of the associated family defines the same Ising model.

\subsection*{Acknowledgments}

N.C.~Affolter and N.~Smeenk were supported by the Deutsche Forschungsgemeinschaft (DFG) Collaborative Research Center TRR 109 ``Discretization in Geometry and Dynamics''.  N.C.~Affolter was also supported by the ENS-MHI chair funded by MHI. 
F.~Dellinger and C.~M{\"u}ller gratefully acknowledge the support by the Austrian Science Fund (FWF) through grant I~4868 (SFB-Transregio “Discretization in Geometry and Dynamics”) and project F77 (SFB “Advanced Computational Design”) . D.~Polly was supported by the French National Agency for Research (ANR) via the grant ANR-18-CE40-0033 (DIMERS) for an inspiring research visit to Paris in the fall of 2021.

N.C.~Affolter would like to thank Dmitry Chelkak for pushing for the development of a discrete theory of maximal s-embeddings, as well as him and Misha Bashok and Rémy Mahfouf for countless explanations and discussions of the Lorentz lift. We would also like to thank Jan Techter, Cédric Boutillier, Carl Lutz, Paul Melotti, Sanjay Ramassamy, Marianna Russkikh and Boris Springborn.

\section{Smooth maximal surfaces}
\label{sec:smMaximal}

Before we develop our discrete theory, we give some reminders of results from the smooth theory. 
These serve as motivation for the discrete setting in the upcoming sections. 

An immersed surface in $\lor$ is 
called \emph{spacelike} if the induced metric on its tangent planes is 
positive definite. For a spacelike immersion, the Gauß map $N$ takes values
in the (upper half of the) spacelike unit sphere $\unip$. 

We call a spacelike surface \emph{isothermic} if it admits conformal curvature
lines. That is, there exists an \emph{isothermic immersion} $f: (u,v)\mapsto f(u,v)$ such that
\begin{align}
	N_u =\kappa_1 f_u, ~N_v = \kappa_2 f_v, ~\textrm{and } \|f_u\|^2 = \|f_v\|^2=:E,
\end{align}
for some functions $\kappa_1, \kappa_2$, which we call the \emph{principle curvatures}. 
The \emph{mean curvature} is given by the arithmetic mean of $\kappa_1$ and $\kappa_2$.

As in Euclidean space, (spacelike) isothermic surfaces in $\lor$ are preserved by
(Lorentz) M\"{o}bius transformations, in the sense that any M\"{o}bius transformation maps isothermic
surfaces to isothermic surfaces. Furthermore, they can be characterized by the existence of a certain dual isothermic surface.

\begin{theorem}\label{prop:smChristoffel}
	Let $f: U\subset \R^2 \to \lor$ be an immersion on an open subset of $\R^2$. Then, $f$ is an isothermic immersion if and only if there exists 
	an immersion $f^\ast: U\to \lor$ such that 
	\begin{align}\label{eq:smChristoffel}
		f^\ast_u = \frac{f_u}{E}, \quad f^\ast_v = -\frac{f_v}{E}.
	\end{align}
	Moreover, $f^*$ is isothermic and is called the \emph{Christoffel dual}.
\end{theorem}

A \emph{maximal surface} is a spacelike immersion with vanishing mean curvature. 
These surfaces maximize local area. They are the natural counterparts
of minimal surfaces in Euclidean space and share many of their properties. Most importantly, we have the following 
facts (see \cite{pember2020}).

\begin{theorem}\label{thm:smMaximalDual}
	All maximal surfaces are isothermic. A spacelike isothermic immersion is maximal if and only if
	its Christoffel dual is contained in a spacelike sphere. Furthermore, the Christoffel dual is the Gauß
	map of the maximal surface (up to scaling and translation). 
\end{theorem}

Theorem~\ref{thm:smMaximalDual} may be used to derive the Weierstraß representation of maximal surfaces
\cite[Thm~1.1]{kobayashi1983}, based on the observation that the Gauß map $N$ of a maximal surface parametrizes (a part of) the spacelike unit sphere $\unip$ isothermically. Moreover, $N$ corresponds to a holomorphic map $g$ via the inverse of the stereographic projection
\begin{align}
	\sigma: \C \to \unip: z \mapsto  \frac{1}{1-|z|^2} \operatorname{Re} \begin{pmatrix}
		2z \cr -2iz \cr 1+|z|^2
	\end{pmatrix}.
\end{align}
Conversely, any holomorphic map induces the Gauß map of a
maximal surface. We can integrate the Christoffel equations \eqref{eq:smChristoffel}, which recovers the Weierstraß formula
\begin{align}\label{eq:smWFormula}
	f(z) = \operatorname{Re}\int G(\xi) d\xi,
\end{align}
with
\begin{align}
	G(z) = \frac{1}{g'(z)} \begin{pmatrix}
		1+g^2(z) \cr i(1-g^2(z)) \cr 2g(z)
	\end{pmatrix}.
\end{align}
Here the surface $f$ is parametrized using the complex coordinate $z=u+iv$. 
Given the holomorphic function $g$, Equation~\eqref{eq:smWFormula} is a conformal curvature line parametrization of $f$. 
All maximal surfaces can be parametrized in this way. 

Equation~\eqref{eq:smWFormula} can also be used to define the associated family of a minimal surface: given a maximal surface $f$ parametrized as in Equation~\eqref{eq:smWFormula}, we define its \emph{associated family} $f^\varphi$ with $\varphi \in \S^1$ via
\begin{align}\label{eq:smAssocFam}
	f^\varphi(z) = \operatorname{Re}
	\int G(\xi)e^{i\varphi} d\xi.
\end{align}
All members of the associated family of a maximal surface are maximal themselves. Let us give a geometric interpretation of Equation~\eqref{eq:smAssocFam}. The directions of coordinate lines on the maximal surface parametrized by $f=f^0$ are given by 
$f_u=\operatorname{Re} G$ and $f_v=\operatorname{Im} G$. Moreover, $f$ is a conformal parametrization, 
which means that $f_u$ and $f_v$ are orthogonal and have equal lengths. The coordinate directions of $f^\varphi$ 
are then given by rotating $\operatorname{Re} G$ and $\operatorname{Im} G$ by the angle $\varphi$ in the tangent 
plane. Thus, all surfaces in the associated family are parametrized by conformal coordinates, but not necessarily along curvature 
lines. For $\varphi=\frac{\pi}{2}$ we obtain conformal asymptotic coordinates. 


\section{Isothermic congruences} \label{sec:isocc}

Let us discuss the basic geometric objects from the literature that we will need in the remainder of the paper. We consider the bipartition of $\Z^2$ into black vertices $\Zb^2$ and white vertices $\Zw^2$, that is
\begin{align}
	\Zb^2 &= \{z \in \Z^2\ | \ z_1 + z_2 \in 2 \Z  \}, & \Zw^2 &= \{z \in \Z^2\ | \ z_1 + z_2 \in 2 \Z + 1 \}.
\end{align}
Note that $\Zb^2 \simeq \Zw^2 \simeq \Z^2$, therefore we may also speak of black (white) edges which are edges of $\Zb^2$ ($\Zw^2$). There is a natural correspondence between the black edges $E(\Zb^2)$ and the white edges $E(\Zw^2)$ or the faces $F(\Z^2)$. Similarly, black vertices $\Zb^2$ and white faces $F(\Zw^2)$ correspond.

In Lorentz space $\lor = \R^{2,1}$, we denote by $\osp_-(\lor)$ the set of oriented timelike spheres, by $\sp_0(\lor)$ the set of null-spheres, and by $\oli_0(\lor)$ the set of oriented isotropic lines. See \cite[Section~3]{admpsiso} for a quick introduction to Lorentz sphere geometry.

\begin{definition} \label{def:nullcongruence}
	A \emph{null congruence} $c$ is the triple of maps
	\begin{align}
		\ccw&: \Zw^2 \rightarrow \osp_-(\lor), & 
		\ccb&: \Zb^2 \rightarrow \sp_0(\lor), & 
		\ccf&: F(\Z^2) \rightarrow \oli_0(\lor),
	\end{align}
  such that the oriented sphere $\cc(v)$ is in oriented contact with the oriented isotropic line $\ccf(f)$ whenever $v$ and $f$ are incident, see Figure~\ref{fig:nullcongruence}.
	The \emph{center net} (of a null congruence $\cc$) is the map $\ccm: \Z^2 \rightarrow \lor$, such that $\ccm(v)$ is the center of the sphere $\cc(v)$. 
\end{definition}
Null congruences were introduced as a special case of the so called \emph{contact congruences} in \cite{admpsiso}. Note that we wrote that $c$ is the triple of maps $\ccw,\ccb,\ccf$, although technically we treat $c$ as the union of the maps $\ccw \cup \ccb \cup \ccf$ and consequently $\ccw,\ccb,\ccf$ as restrictions of $\cc$.

Next, we want to introduce a special class of null congruences, for which we need the classical notion of a (discrete) conjugate net, see \cite{sauerqnet, dsqnet, bsddgbook}.

\begin{figure}
	\centering
	\includegraphics[height=.45\textwidth]{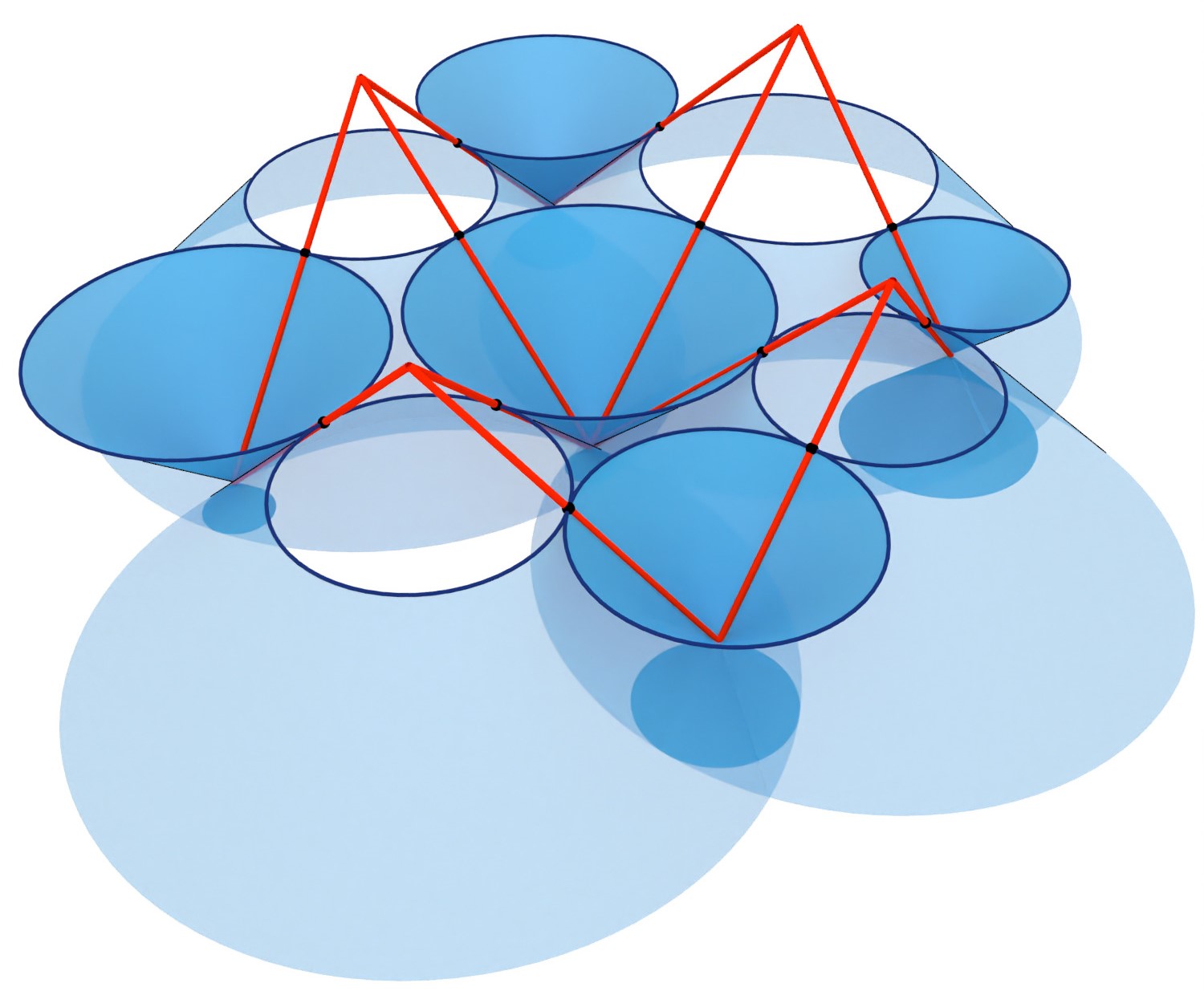}
	\includegraphics[height=.45\textwidth]{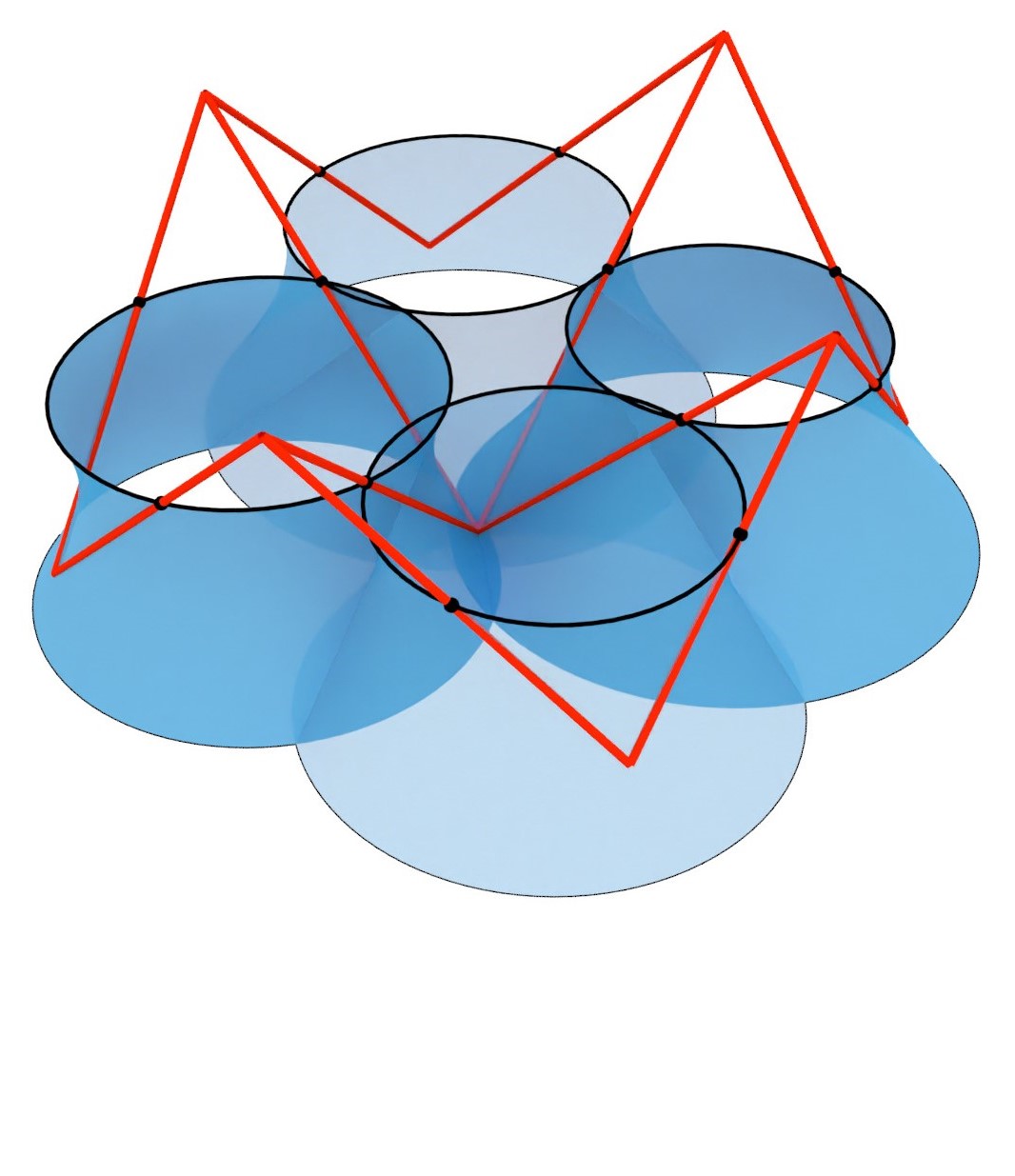}
	\caption{The null-spheres $\ccb$ (left) and time-like spheres $\ccw$ (right) of a null congruence $\cc$ as well as the isotropic lines $\ccf$ (red).}
	\label{fig:nullcongruence}
\end{figure}

\begin{definition}\label{def:conjugatenet}
	A map $x: \Z^2 \rightarrow \R^n$ is a \emph{conjugate net} if the image of the vertices of each face $f\in F(\Z^2)$ are contained in a plane. We denote this plane by $x(f)$. 
\end{definition}

Moreover, we call a conjugate net spacelike if all the face-planes are spacelike.

\begin{definition}\label{def:isothermiccongruence}
	A \emph{(spacelike Lorentz) isothermic congruence} is a null congruence $\cc$ such that $\ccmw$ is a spacelike conjugate net.
\end{definition}

Isothermic congruences were introduced in \cite{admpsiso}, 
and are of particular use because they are in correspondence with the following known object from discrete differential geometry. Let us denote by $\ci_+(\lor)$ the set of spacelike circles (which are intersections of spheres with spacelike planes).

\begin{definition}\label{def:sisothermic}
	A \emph{(spacelike Lorentz) S-isothermic net} $h$ is a collection of maps 
	\begin{align}
		\isow&: \Zw^2 \rightarrow \osp_-(\lor), &
		\isob&: \Zb^2 \rightarrow \ci_+(\lor), &
		\isof&: F(\Z^2) \rightarrow \lor,
	\end{align}
  such that for all faces $f = (w,b,w',b') \in F(\Z^2)$ the two spheres $\isow(w)$, $\isow(w')$ are in oriented contact and the two circles $\isob(b)$, $\isob(b')$ intersect the two spheres orthogonally in the point $\isof(f)$, see Figure~\ref{fig:isothermic} (left).
	The \emph{center net} of an S-isothermic net $\iso$ is the map $\isom: \Z^2 \rightarrow \lor$, such that $\isom(v)$ is the center of $\iso(v)$ for all $v\in \Z^2$.
\end{definition}

The original definition of S-isothermic nets was given for Euclidean space \cite{bpisoint}, which was subsequently translated to Lorentz space in \cite{bhscmc,admpsiso}. The correspondence to isothermic congruences is given in the next theorem, for a proof see \cite{admpsiso}.

\begin{theorem}\label{th:isocongruencetoisonet}
	Every isothermic congruence $\cc$ defines a unique S-isothermic net $\iso$ such that $\ccw = \isow$, vice versa every S-isothermic $\iso$ net defines two isothermic congruences $\cc^1,\cc^2$ such that $\isow = \ccw^1 = \ccw^2$.
\end{theorem}

The circles $\isob$ are obtained from the spheres $\ccb$ by intersecting each null-sphere $\ccb(b)$ with the spacelike plane spanned by the adjacent centers of $\ccw$.

\section{Isothermic incircular nets}  \label{sec:isocp}

\begin{figure}
	\centering
	\includegraphics[height=.4\textwidth]{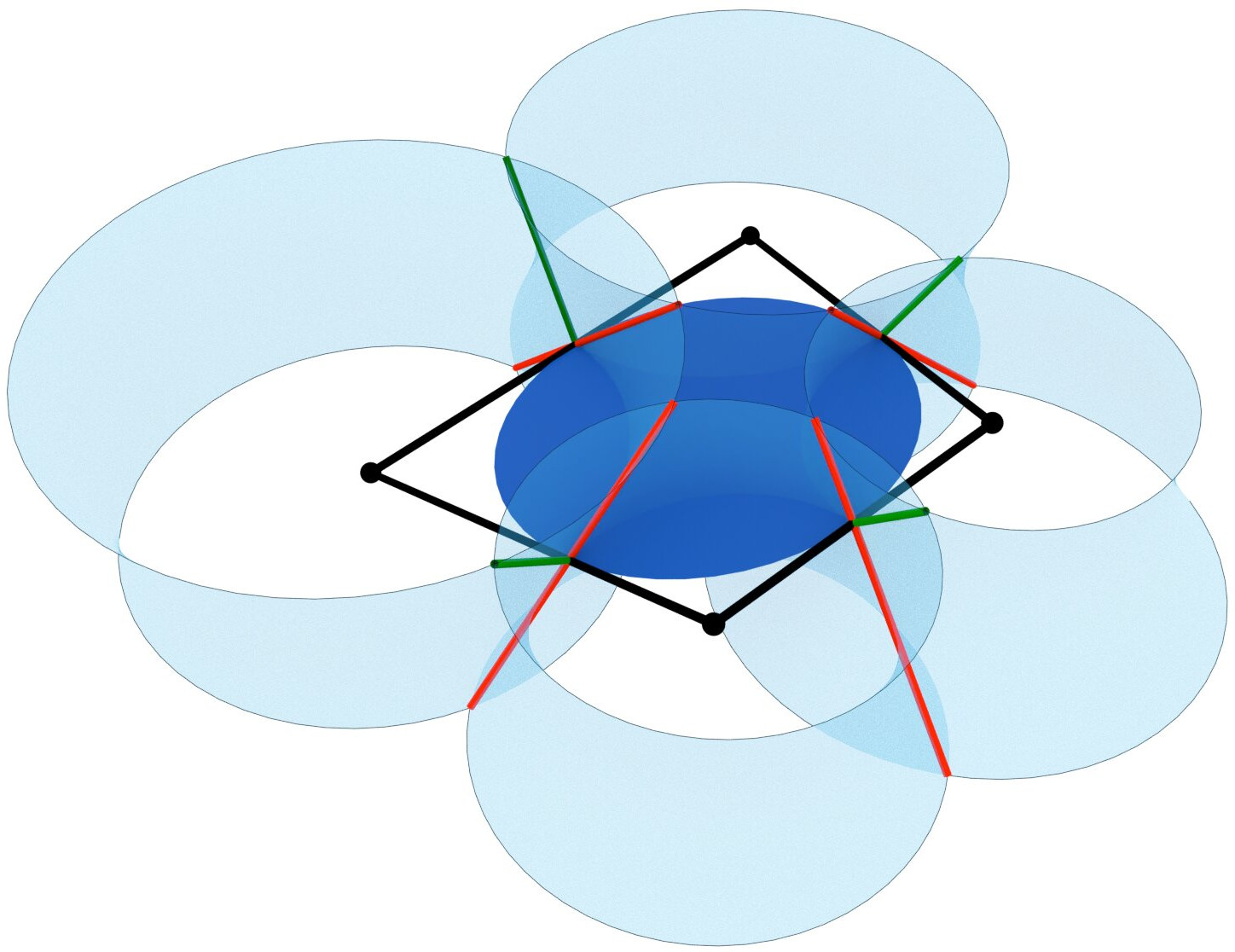}
	\hspace{5mm}
	\includegraphics[height=.41\textwidth]{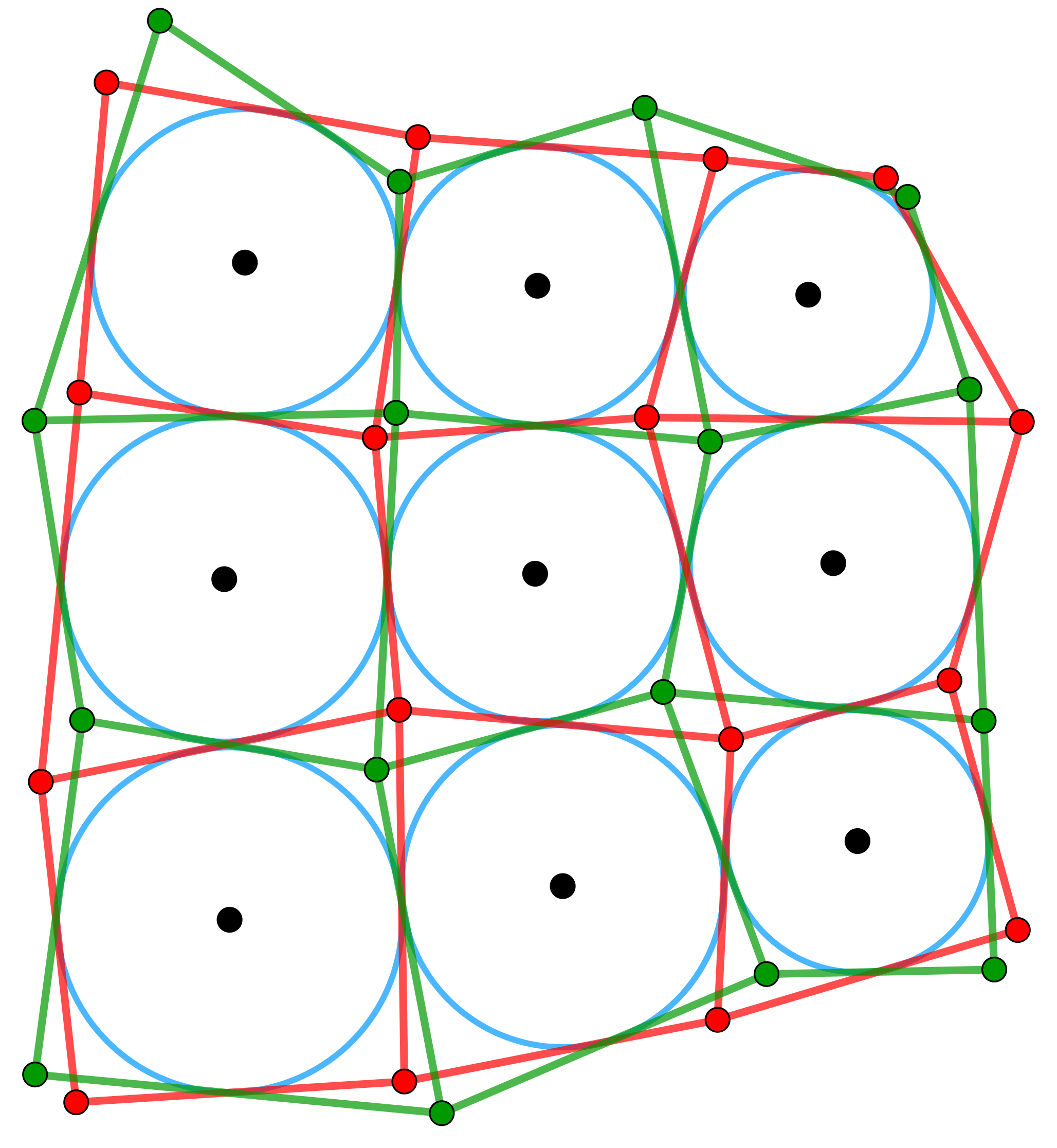}
	\caption{Left: a circle of $\isob$ (blue) of an S-isothermic net $\iso$ and the four adjacent timelike Lorentz spheres of $\isow$ (teal), as well as the isotropic lines of $\ccf^1$ (green) and $\ccf^2$ (red) of the two corresponding isothermic congruences $\cc^1$, $\cc^2$.
		Right: An isothermic incircular net $p^1$ (green) and the second incircular net $p^2$ (green) given by the other tangents.
	}
	\label{fig:isothermic}
\end{figure}

In \cite{admpsiso}, we showed how null congruences correspond to incircular nets and how we identify the special case of incircular nets that correspond to isothermic congruences. Let us briefly revisit the definitions and theorems that we use in the remainder of the paper.

We begin with incircular nets, which are essentially quad nets in $\eucl$ such that every quad has an incircle, we formalize this in the following definition.
\begin{definition} \label{def:incircular}
	An \emph{incircular net} $\cp$ is a triple of maps
	\begin{align}
		\cpw&: \Zw^2 \rightarrow \oci(\eucl), & 
		\cpb&: \Zb^2 \rightarrow \eucl, & 
		\cpf&: F(\Z^2) \rightarrow \oli(\eucl),
	\end{align}
	such that the oriented line $\cpf(f)$ contains $\cpb(b)$ and is in oriented contact with $\cpw(w)$ whenever $w\in \Zw^2$, $b\in\Zb^2$ are incident with $f$.
\end{definition}

Each null congruence $\cc$ defines an incircular net as follows: For each face $f$ the line $\cpf(f)$ is the orthogonal projection of $\ccf(f)$, $\cpb(b)$ is the orthogonal projection of the apex of $\ccb(b)$, and $\cpw(w)$ is the projection of the smallest circle of $\ccw(w)$ (the contour of $\ccw(w)$ under orthogonal projection). Every incircular net comes from a null congruence in this way, see \cite{admpsiso} for details.

The quantities that define the statistics of the Ising model from an incircular net are the \emph{$X$-va\-ri\-a\-bles} \cite{chelkaksembeddings}.

\begin{definition}\label{def:xvar}
	Let us identify $\eucl \simeq \C$ and consider an incircular net $\cp$. The $X$-variables $X_\circ: \Zw^2\rightarrow \R$ are defined as
	\begin{align}	
		X_\circ(w) = -\frac{(\cpb(b_1)-\cpmw(w))(\cpb(b_2)-\cpmw(w))}{(\cpb(b_3)-\cpmw(w))(\cpb(b_4)-\cpmw(w))},
	\end{align}	
	for all $w\in \Zw^2$ and $b_1,b_2,b_3,b_4$ adjacent to $w$ in counterclockwise order, with $b_1 = w + (1,0)$.
\end{definition}
Note that there is a general definition of $X$-variables for conical nets, in which case the $X$-variables are defined at all vertices of $\Z^2$. In the special case of incircular nets, it suffices to consider the $X$-variables at white vertices. The $X$-variables are real because $\cp$ is a special case of a conical net, see \cite[Remark~11.2]{admpsiso}.

Let $\cc$ be a null congruence that projects to an incircular net $\cp$. Then the $X$-variables of $\cp$ may also be read off of $\cc$ as expressed in the next lemma.

\begin{lemma} \label{lem:xcyclo}
	Let $\cc$ be a null congruence that projects to an incircular net $\cp$. Then the $X$-variables of $\cp$ satisfy
	\begin{align}\label{eq:xvarLaguerre}
		X_\circ(w) = \frac{|\ccm(b_1)-\ccm(b_3)|^2_\lor}{|\ccm(b_2)-\ccm(b_4)|^2_\lor},
	\end{align}
	for all $w\in \Zw^2$ and $b_1,b_2,b_3,b_4$ adjacent to $w$ in counterclockwise order, with $b_1 = w + (1,0)$.
\end{lemma}

Next, let us characterize which incircular nets are the projections of isothermic congruences. Note that the set of oriented lines that are tangent to two (different) oriented circles in $\eucl$ has at most two elements. In particular, if the two circles are disjoint as disks then there are exactly two such lines. Thus given one tangent, we may speak of the \emph{other} tangent. If there is only one tangent, then the other tangent refers to the same tangent.

\begin{definition}\label{def:isocp}
	An \emph{isothermic incircular net} is an incircular net $\cp^1$ such that there is a second incircular net $\cp^2$ that consists of the other tangents, see Figure~\ref{fig:isothermic} (right).
\end{definition}

By definition, $\cp^2$ is also an isothermic incircular net which has the same incircles as $\cp^1$.


\section{Koebe congruences} \label{sec:koebe}

We briefly revisit circle patterns in the hyperbolic plane. For our purposes the \emph{hyperboloid model} of the hyperbolic plane is particularly practical. Let $\unip$ be the (upper half of the) spacelike unit-sphere in $\lor$. We identify the hyperbolic plane $\mathcal H$ with $\unip$.

\begin{definition}\label{def:hocp}
	A \emph{hyperbolic orthogonal circle pattern} $\Cp$ is a pair of maps
	\begin{align}
		\Cp: \Z^2 &\rightarrow \ci_+(\unip), & \Cpf: F(\Z^2) &\rightarrow \unip,
	\end{align}
	such that 
	\begin{enumerate}
		\item around each face $f$ the four circles intersect in the point $\Cp(f)$, and
		\item adjacent circles intersect orthogonally.\qedhere
	\end{enumerate}
\end{definition}

Each circle $\Cp(v)$ in the hyperbolic plane corresponds to the intersection of $\unip$ with a plane, and each such plane is in bijection with the polar point $\Cp^\perp(v)$ outside of $\unip$.
The point $\Cp^\perp(v)$ is the center of a (unique) timelike sphere $S(v)$ that intersects $\unip$ orthogonally in the circle $\Cp(v)$.
This allows us to define an S-isothermic net $\iso$ by
\begin{align}
	\isow(w) &= S(w), & \isob(b) &= \Cp(b), & \isof(f) = \Cpf(f),
\end{align}
We may choose the orientation of one initial sphere of $h$ freely, which then determines the orientations of all other spheres.
This motivates the following definition.

\begin{definition}
  \label{defn:koebe-isothermic}
	A \emph{Koebe isothermic net} is an S-isothermic net $\iso$ such that every circle of $\isob$ is contained in the upper half of the unit-sphere $\unip$, see Figure~\ref{fig:maximal} (left). A \emph{Koebe congruence} is an isothermic congruence that corresponds to a Koebe isothermic net via Theorem~\ref{th:isocongruencetoisonet}.
\end{definition}

It is not hard to see that Koebe isothermic nets are exactly those isothermic nets that correspond to hyperbolic orthogonal circle patterns. An advantage of hyperbolic orthogonal circle patterns (and therefore Koebe isothermic nets) is that they can be constructed from boundary data as the minimizer of a certain convex functional, see \cite{bsvariationalcp}. In fact, existence and uniqueness is guaranteed, and this also works for combinatorics that are more general than $\Z^2$.

\begin{remark}
  Let $\iso$ be a Koebe isothermic net. At each face $f$ there is a unique tangent $T(f)$ to the two adjacent circles of $\isob$, which is also a tangent to $\unip$. Around a white vertex $w$, the adjacent tangents intersect in the point $\isomw(w)$. Viewed in this way, as an edge-tangent conjugate net, $\isomw$ is called a \emph{Koebe polyhedron} \cite{bsvariationalcp, bhsminimal}. Note that after a projective transformation, $\isomw$ is still a Koebe polyhedron, edge-tangent to a new quadric which is the image of $\unip$. In particular, there are projective transformations that map $\unip$ to $\S^2 \subset \R^3$. There are also projective transformations that fix $\unip$ or fix $\S^2$, which are the corresponding Möbius transformations of $\unip$ or $\S^2$, respectively. In particular, this enables us to start with a Euclidean Koebe polyhedron and map it to a Lorentz Koebe polyhedron that is edge tangent to only the upper half of $\unip$. Thus, the difference between treating Lorentz maximal surfaces and Euclidean minimal surfaces lies only in the treatment of the boundary conditions, which are given in the hyperbolic plane $\mathcal H^2$ or the Euclidean plane $\eucl$, respectively.
\end{remark}

\begin{remark} \label{rem:koebecp}
	Which isothermic incircular nets are projections of Koebe congruences? Recall that isothermic incircular nets come in pairs $\cp^1,\cp^2$ (Definition~\ref{def:isocp}) which share the same incircles. For every black vertex $b\in\Zb^2$ we may define the line $\ell(b)$ spanned by $\cpb^1(b)$ and $\cpb^2(b)$. An isothermic incircular net comes from a Koebe congruence if and only if all the lines $\ell(b)$ intersect in a point. This can be shown by observing that each line $\ell(b)$ is the projection of the axis of the corresponding circle $\isob$ in the Koebe isothermic net $\iso$. If the axes all intersect in a point, this point is the center of the unit-sphere $\unip$. Also note that if all the lines are parallel or are all points (in the degenerate case that $\cp^1=\cp^2$), then the circles of $\iso$ are all contained in a plane. Thus $\iso$ is essentially flat and is not a Koebe congruence.
\end{remark}

\section{Christoffel dual} \label{sec:christoffel}

\begin{figure}
	\centering
	\begin{tikzpicture}[line cap=round,line join=round,>=triangle 45,x=1.0cm,y=1.0cm,scale=1.5]
		\definecolor{qqzzqq}{rgb}{0.5,0.25,0.1}
		\definecolor{ttttff}{rgb}{0.2,0.2,1.}
		\clip(-6.49007331469685,-1.6167843497060181) rectangle (1.6,2.0010235341768725);
		\draw [line width=1.5pt,blue!80!black] (-4.24,0.02) circle (0.9987414173774751cm);
		\draw [line width=1.5pt,blue!80!black] (-0.42746692907739503,-0.12274082086700801) circle (1.001260168648888cm);
		\draw [line width=1.5pt] (-4.44273807014855,-1.2625517400125128)-- (-3.3044103136325047,-0.5856276320078151);
		\draw [line width=1.5pt] (-3.3044103136325047,-0.5856276320078151)-- (-3.0921841347984063,1.2569796075931026);
		\draw [line width=1.5pt] (-3.0921841347984063,1.2569796075931026)-- (-6.142269373424083,0.6722414186572291);
		\draw [line width=1.5pt] (-6.142269373424083,0.6722414186572291)-- (-4.44273807014855,-1.2625517400125128);
		\draw [line width=1.5pt] (-1.9750422636041638,0.12189058615006641)-- (-0.8016207773583474,-1.213964693154792);
		\draw [line width=1.5pt] (-0.8016207773583474,-1.213964693154792)-- (0.4830962232332596,-0.967668923217702);
		\draw [line width=1.5pt] (0.4830962232332596,-0.967668923217702)-- (0.7985574922498841,1.771254339204683);
		\draw [line width=1.5pt] (0.7985574922498841,1.771254339204683)-- (-1.9750422636041638,0.12189058615006641);
		\draw [line width=1.5pt,color=qqzzqq] (-4.428046295555406,1.0008786925573707)-- (-4.24,0.02);
		\draw [line width=1.5pt,color=qqzzqq] (-4.24,0.02)-- (-4.990362066814596,-0.6391215270880599);
		\draw [line width=1.5pt,color=qqzzqq] (-3.7295229457086774,-0.8384273969458466)-- (-4.24,0.02);
		\draw [line width=1.5pt,color=qqzzqq] (-4.24,0.02)-- (-3.2478179286600612,-0.09427666470787437);
		\draw [line width=1.5pt,color=qqzzqq] (-1.1797213529895585,-0.7835246032250336)-- (-0.42746692907739503,-0.12274082086700801);
		\draw [line width=1.5pt,color=qqzzqq] (-0.42746692907739503,-0.12274082086700801)-- (-0.9392313683786432,0.7378514658689811);
		\draw [line width=1.5pt,color=qqzzqq] (-0.42746692907739503,-0.12274082086700801)-- (0.5672173513530837,-0.23730568278941422);
		\draw [line width=1.5pt,color=qqzzqq] (-0.42746692907739503,-0.12274082086700801)-- (-0.23894639480736937,-1.1060932162377335);
		\begin{normalsize}
		\draw [fill=ttttff] (-4.24,0.02) circle (2.5pt);
		\draw[color=ttttff] (-3.85,0.2785518624988409) node {$\isomb(b)$};
		\draw [fill=qqzzqq] (-4.428046295555406,1.0008786925573707) circle (2.5pt);
		\draw[color=qqzzqq] (-4.304571623750529,1.2540017697573678) node {$\isof(f_4)$};
		\draw [fill=qqzzqq] (-4.990362066814596,-0.6391215270880599) circle (2.5pt);
		\draw[color=qqzzqq] (-5.418284195106618,-0.733940446301149) node {$\isof(f_1)$};
		\draw [fill=qqzzqq] (-3.7295229457086774,-0.8384273969458466) circle (2.5pt);
		\draw[color=qqzzqq] (-3.2513761257407803,-0.9561948555499271) node {$\isof(f_2)$};
		\draw [fill=qqzzqq] (-3.2478179286600612,-0.09427666470787437) circle (2.5pt);
		\draw[color=qqzzqq] (-2.8,0.10568732197201337) node {$\isof(f_3)$};
		\draw [fill=black] (-6.142269373424083,0.6722414186572291) circle (2.5pt);
		\draw[color=black] (-6.0208695618383175,0.9206201558842003) node {$\isomw(w_1)$};
		\draw [fill=black] (-3.0921841347984063,1.2569796075931026) circle (2.5pt);
		\draw[color=black] (-2.971045168257859,1.5132985805476091) node {$\isomw(w_4)$};
		\draw [fill=black] (-4.44273807014855,-1.2625517400125128) circle (2.5pt);
		\draw[color=black] (-5.05,-1.2895764694230947) node {$\isomw(w_2)$};
		\draw [fill=black] (-3.3044103136325047,-0.5856276320078151) circle (2.5pt);
		\draw[color=black] (-2.7,-0.5) node {$\isomw(w_3)$};
		\draw [fill=ttttff] (-0.42746692907739503,-0.12274082086700801) circle (2.5pt);
		\draw[color=ttttff] (-0.15,0.13038225633298875) node {$\isomb^*(b)$};
		\draw [fill=qqzzqq] (0.5672173513530837,-0.23730568278941422) circle (2.5pt);
		\draw[color=qqzzqq] (1.1,-0.005439882652375758) node {$\isof^*(f_3)$};
		\draw [fill=qqzzqq] (-0.23894639480736937,-1.1060932162377335) circle (2.5pt);
		\draw[color=qqzzqq] (-0.0200005121213013,-1.4) node {$\isof^*(f_4)$};
		\draw [fill=qqzzqq] (-0.9392313683786432,0.7378514658689811) circle (2.5pt);
		\draw[color=qqzzqq] (-0.9337130834773901,1.0811372292305401) node {$\isof^*(f_2)$};
		\draw [fill=qqzzqq] (-1.1797213529895585,-0.7835246032250336) circle (2.5pt);
		\draw[color=qqzzqq] (-1.65,-0.8574151181060258) node {$\isof^*(f_1)$};
		\draw [fill=black] (-1.9750422636041638,0.12189058615006641) circle (2.5pt);
		\draw[color=black] (-2.1561123343456714,0.45) node {$\isomw^*(w_2)$};
		\draw [fill=black] (-0.8016207773583474,-1.213964693154792) circle (2.5pt);
		\draw[color=black] (-1.4,-1.3266188709645577) node {$\isomw^*(w_1)$};
		\draw [fill=black] (0.4830962232332596,-0.967668923217702) circle (2.5pt);
		\draw[color=black] (1.1,-0.8203727165645627) node {$\isomw^*(w_4)$};
		\draw [fill=black] (0.7985574922498841,1.771254339204683) circle (2.5pt);
		\draw[color=black] (0.2,1.8590276616012642) node {$\isomw^*(w_3)$};
		\end{normalsize}
	\end{tikzpicture}
  \caption{Local Christoffel dual $\iso^*$ of an S-isothermic net $\iso$. Corresponding edges (black and brown) are parallel.}
	\label{fig:christoffeldual}
\end{figure}
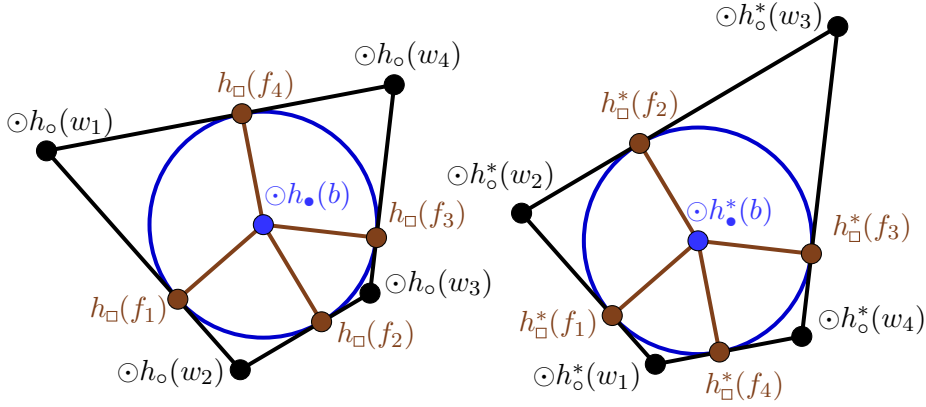

Let us briefly recall the \emph{Christoffel dual} for S-isothermic nets \cite{bhsminimal} with the notation as in \cite{admpsiso}, see also Figure~\ref{fig:christoffeldual}. Let us denote by 
\begin{align}
  \isoc = \isom \cup \isof: \Z^2 \simeq (\Z^2 \cup F(\Z^2)) \rightarrow \lor, \label{eq:combinediso}
\end{align}
the combination of center net $\isom$ and contact points $\isof$ of an S-isothermic net $\iso$, defined on $\Z^2 \cup F(\Z^2)$, which we identify with $\Z^2$ for the time being. In this notation, we define the discrete differential
\begin{align}
	\d \isoc(v,v') = \isoc(v') - \isoc(v),
\end{align}
for all adjacent $v,v' \in \Z^2$. We also define the dual differential
\begin{align}
	\d \isoc^*(v,v') = \pm \frac{\d \isoc(v',v)}{|\d \isoc(v',v)|^2},
\end{align}
for all adjacent $v,v'\in \Z^2$, where the sign is $+$ if the edge $(v,v')$ is horizontal or $-$ if the edge is vertical in $\Z^2$. This dual form is \emph{closed}, which means it may be integrated to obtain a map $\isoc^*: \Z^2 \rightarrow \lor$. In fact, it can be shown that $\isoc^*$ defines a unique S-isothermic net $\iso^*$ (up to translation), which is called the \emph{Christoffel dual} of $\iso$. 

The identification in Theorem~{\protect\ref{th:isocongruencetoisonet}} implies that there is also a sensible definition of the Christoffel dual $\cc^*$ of an isothermic congruence $\cc$.
Consider the combined map
\begin{align}
	\ccc = \ccm \cup \isof: \Z^2 \simeq V(\Z^2) \cup F(\Z^2) \rightarrow \lor, \label{eq:combinedcc}
\end{align}	
that is the combination of $\ccm$ and $\isof$. In comparison with $\isoc$ we substituted $\isomb$ with $\ccmb$, therefore $\ccc$ is not a conjugate net. However, for a face $f$ and an adjacent white vertex $w$ or black vertex $b$ it is possible to integrate 
\begin{align}
	\d \ccc^*(f,w) &= \pm \frac{\d \ccc(f,w)}{R(w)^2}, &
	\d \ccc^*(f,b) &= \pm \frac{\d \ccc(f,b)}{r(b)^2}, \label{eq:nullspheredual}
\end{align}
where $R(w)$ is the radius of $\isow(w)$ and $r(b)$ is the radius of the incircle of the corresponding quad of $\ccmw$ at $b$. The difference to the formula for $\d \isoc^*$ is that we cannot divide by $|\d \ccc(f,b)|$, since that length is zero. Instead, this formula works since the triangle $\isof(f),\isomb(b),\ccmb(b)$ is a right-angled triangle and 
\begin{align}
	|\isof(f) - \isomb(b)|^2 = - |\isomb(b) - \ccmb(b)|^2.
\end{align}
The discrete differential $\d \ccc^*$ is consistent in the sense that the integrated net $\ccc^*$ is indeed the combination of $\ccm^*$ and $\isof^*$.

\section{Maximal surfaces} \label{sec:maximal}

\begin{figure}
	\centering
	\includegraphics[width=.475\textwidth]{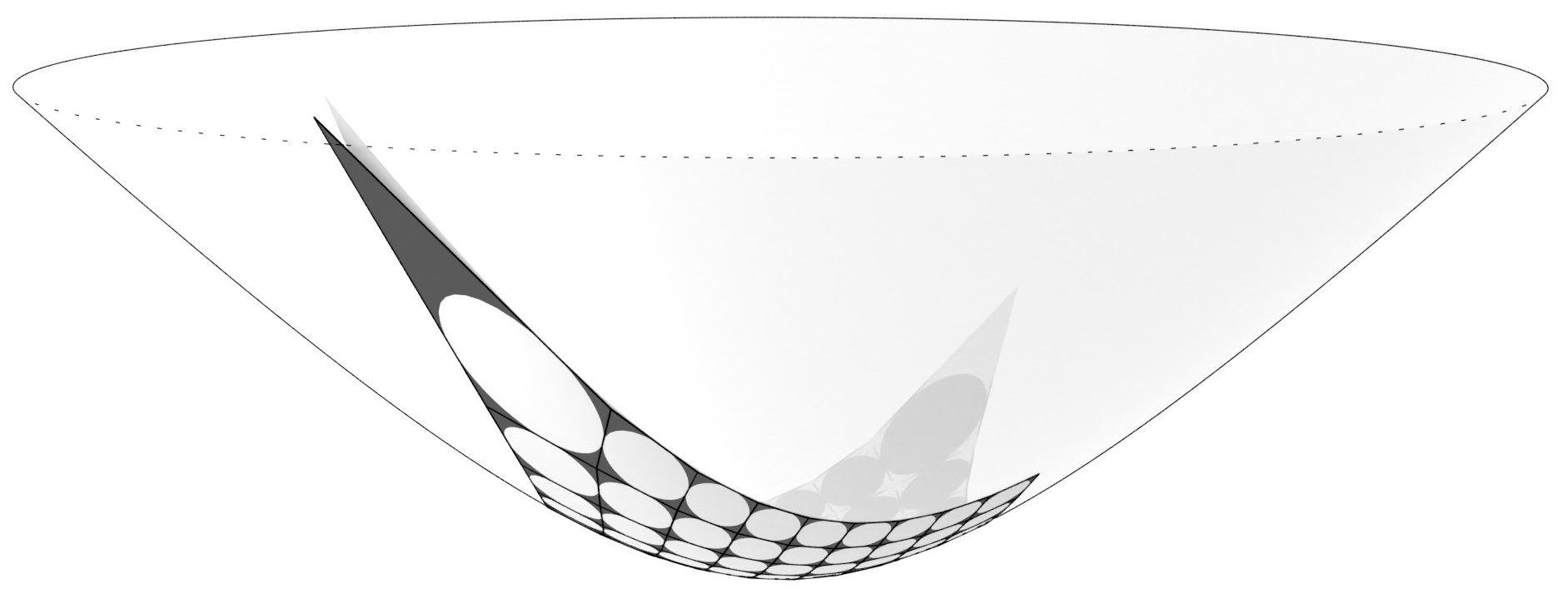}
	\hspace{2mm}
	\includegraphics[width=.475\textwidth]{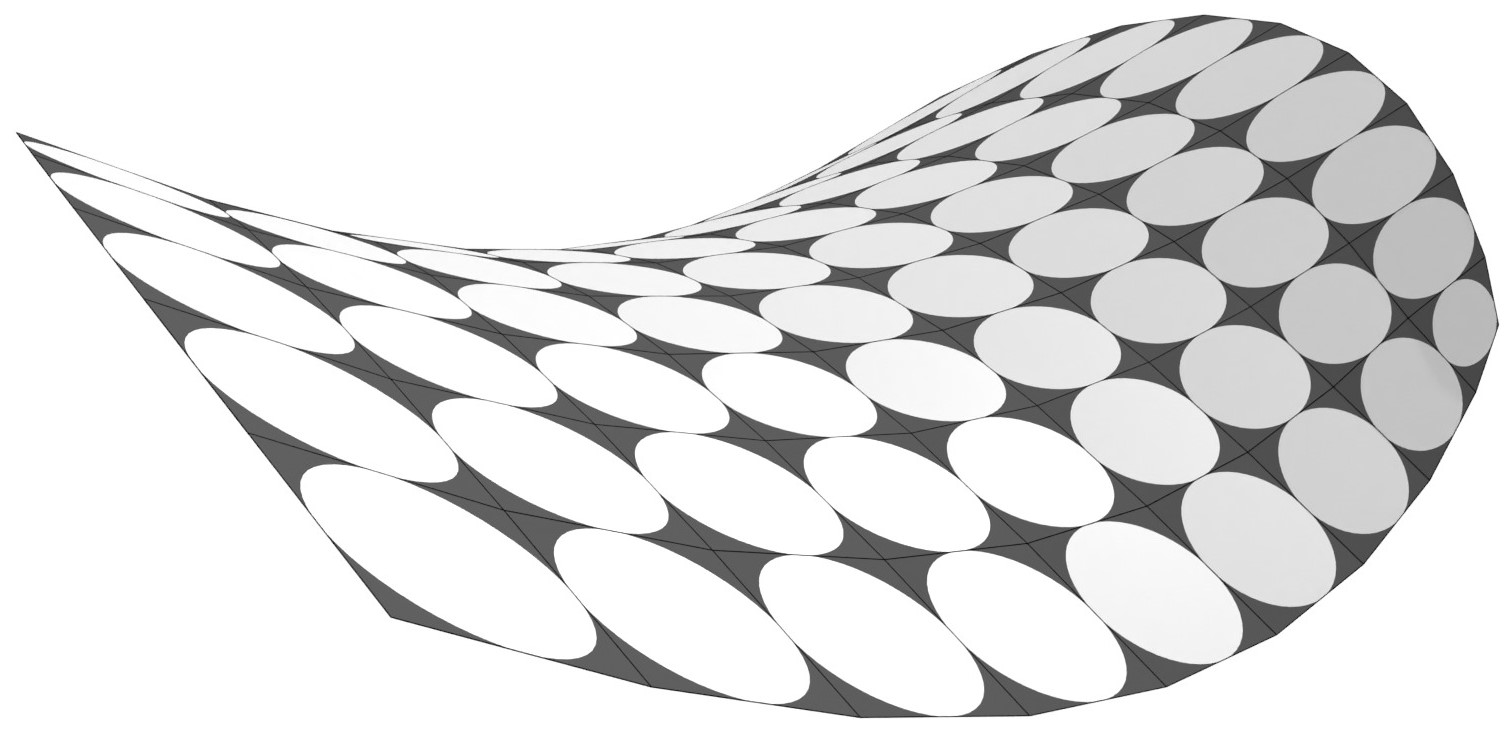}
  \caption{Left: the circles of a Koebe isothermic net. Right: the circles of a discrete maximal surface.}
	\label{fig:maximal}
\end{figure}

\begin{definition}
  A \emph{(discrete spacelike Lorentz) maximal surface} is an S-isothermic net $\iso$ such that the Christoffel dual $\iso^*$ is a Koebe isothermic net (see Figure~\ref{fig:maximal}).
	A \emph{(spacelike Lorentz) maximal congruence} is an isothermic congruence $\cc$ such that the Christoffel dual congruence $\cc^*$ is a Koebe congruence.
\end{definition}

The following direct characterization of discrete maximal surfaces is an immediate translation of the characterization of discrete minimal surfaces in \cite[Definition~6]{bhsminimal}.

\begin{theorem}\label{th:maxcirclecenters}
  An S-isothermic net $\iso$ is a maximal surface if and only if for each white vertex $w$ and its four adjacent vertices $b_1$, $b_2$, $b_3$, $b_4$ the four circle centers $\odot\isob(b_1)$, $\odot\isob(b_2)$, $\odot\isob(b_3)$, $\odot\isob(b_4)$ are contained in a plane.
\end{theorem}

If $\cc$ is a maximal congruence, we understand the Koebe congruence $\cc^*$ as the normal map of $\cc$, in the sense that the vectors from the origin (the center of the unit-sphere $\unip$) to each point $\ccmw^*(w)$ are the normal vectors at $\ccmw(w)$.

Let $t\in \R$ and consider the map 
\begin{align}
	\ccmw^t = \ccmw + t \ccmw^*.
\end{align}
We think of $\ccmw^t$ as an \emph{offset surface} of $\ccmw$, since it is edge- (and face-) parallel to $\ccmw$ at some offset distance $t$. Each face of $\ccmw^t$ has an area $A^t$ that changes quadratically in $t$. This change is captured by the \emph{(discrete) Steiner formula} \cite{schiefunifii, schiefmaxmin, lpwywconical, bpwcurvatureparallel, bsddgbook}
\begin{align}
	A^t = A^0 - 2 H t +  K t^2, 
\end{align}
where we understand $H$ as the \emph{(discrete) mean curvature} and $K$ as the \emph{(discrete) Gauß curvature}. Because $\ccmw$ is a K{\oe}nigs net (see \cite{admpsiso, bsddgbook}), a simple computation shows that $H$ is actually \emph{zero}. This provides another justification for calling maximal congruences maximal, since they have constant discrete mean curvature zero, compare with Section~\ref{sec:smMaximal}.

\begin{remark}
	Which isothermic incircular nets are projections of maximal congruences? Let us give a local characterization which uses four black vertices $b_1,b_2,b_3,b_4$ adjacent to a white vertex $w$ in circular order. Consider the lines $\ell(b_1)$, $\ell(b_2)$, $\ell(b_3)$, $\ell(b_4)$ as in Remark~\ref{rem:koebecp}, as well as the four intersection points $P_{12}$, $P_{23}$, $P_{34}$, $P_{41}$ given by consecutive lines -- that is, $P_{12} = \ell(b_1) \cap \ell(b_2)$ and so on. Let us also denote by $M_i = \frac{1}{2}(\cp^1(b_i)+\cp^2(b_i))$ the midpoint of the two corresponding points in the pair of isothermic incircular nets $\cp^1,\cp^2$, which is on $\ell(b_i)$. Then $\cp^1$, $\cp^2$ are the projection of a maximal congruence if and only if
	\begin{align}
		\frac{M_1 - P_1}{P_1 - M_2}\frac{M_2 - P_2}{P_2 - M_3}\frac{M_3 - P_3}{P_3 - M_4}\frac{M_4 - P_4}{P_4 - M_1} = 1,
	\end{align}
	where we evaluate the equation by identifying $\eucl$ with the complex plane $\C$.
	The equation is a generalization of Menelaus' theorem (see \cite{bsddgbook}) that characterizes the observation in Theorem~\ref{th:maxcirclecenters}, that is, the coplanarity of the four circle centers $\iso(b_i)$. 
\end{remark}


\section{Weierstra{\ss} representation} \label{sec:weierstrass}

The geometric idea of the \emph{discrete Weierstraß representation} is as follows. First of all, we understand an orthogonal circle pattern in $\R^2$ as a discrete conformal map \cite{schrammocp}. If restricted to the unit disc, the orthogonal circle pattern is a hyperbolic orthogonal circle pattern in the Poincaré disk model of the hyperbolic plane. Then, we lift the orthogonal circle pattern via inverse stereographic projection to the upper half of the timelike unit-sphere $\unip$ to obtain a hyperbolic orthogonal circle pattern in $\unip$ (as in Definition~\ref{def:hocp}). Then we take the Christoffel dual to obtain the corresponding discrete maximal surface. Each operation may be expressed via formulas, combining them all yields the discrete Weierstraß representation, analogous to the smooth setting (see Section~\ref{sec:smMaximal}).

Compared to the discrete Weierstraß representation for minimal surfaces in Euclidean space \cite{bhsminimal}, we only need to adapt some of the signs in the formula to the Lorentz setup. Let $\Cp$ be a hyperbolic orthogonal circle pattern in the Poincaré disk model, that is, in the unit-disk in $\R^2 \simeq \C$. The points $\Cpf: F(\Z^2) \rightarrow \mathcal H^2$ are the intersection points of $\Cp$ which are contained in the unit disk $\mathcal H^2 \subset \C$. Let us also denote by $\rho(v)$ the radius and by $\Cpm(v)$ the center of the circle $\Cp(v)$. Let $\iso$ be the Koebe isothermic net that stereographically projects to $\Cp$, and $\iso^*$ the corresponding maximal surface. Let $w,w'$ be two white vertices adjacent to a common face $f$. Then
\begin{align}
	\isomw^*(w) - \isomw^*(w') = \pm \mathrm{Re}\!\left[\frac{R_\circ^*(w) + R_\circ^*(w')}{1-|\Cpf(f)|^2}\frac{\overline{\Cpmw(w)} - \overline{\Cpmw(w')}}{|\Cpmw(w) - \Cpmw(w')|} 
	\begin{pmatrix}
		\phantom{i(}1 + |\Cpf(f)|^2\phantom{)} \\ i(1 - |\Cpf(f)|^2) \\ 2 \Cpf(f)
	\end{pmatrix}\right]\!, \label{eq:weierstrass}
\end{align}
where $R^*_\circ$ are the radii of the spheres of $\isow^*$, which may be expressed as
\begin{align}
	R_\circ^*(w) = \frac{1 - |{\Cpmw(w)}|^2 + \rho_\circ(w)^2}{2 \rho_\circ(w)}. \label{eq:weierstrassradius}
\end{align}

Now, let $\cc^*$ be one of the two null congruences that correspond to $\iso^*$, as in Theorem~\ref{th:isocongruencetoisonet}. Additionally, let $\cp^*$ be the incircular net corresponding to $\cc^*$. We may restrict Equation~\eqref{eq:weierstrass} to the first two coordinates to obtain an explicit differential formula for $\odot\cpw^*$. Thus, we know the incircle centers of $\cp^*$. Moreover, the incircle radii (the radii of $\cpw^*$) are simply the radii $R^*_\circ$. Hence, the incircles $\cpw^*$ are determined. By taking the tangents $\cpf^*$ to the incircles in a consistent manner, we also obtain the whole of $\cpf^*$, and therefore also $\cpb^*$.

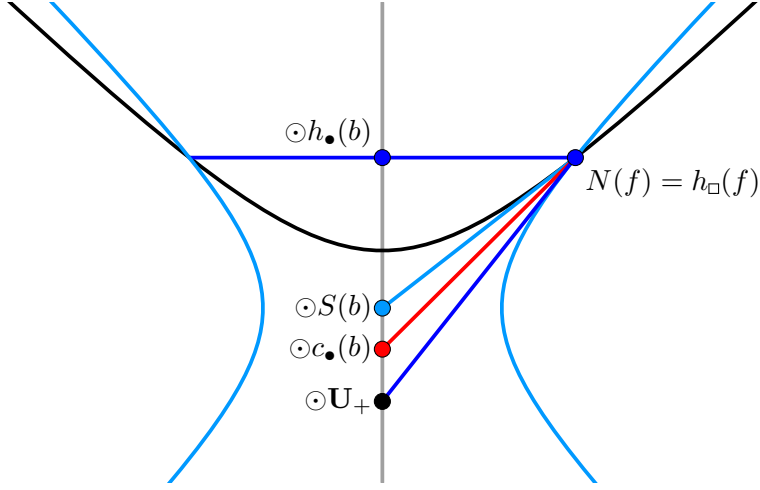
\begin{figure}
	\centering
	\begin{tikzpicture}[line cap=round,line join=round,>=triangle 45,x=1.0cm,y=1.0cm,scale=2]
		\definecolor{yqqqqq}{rgb}{1,0.,0.}
		\definecolor{qqqqff}{rgb}{0.,0.,1.}
		\definecolor{qqzzff}{rgb}{0.,0.6,1.}
		\definecolor{aqaqaq}{rgb}{0.6274509803921569,0.6274509803921569,0.6274509803921569}
		\clip(-2.544311505038312,-0.5360992384404077) rectangle (2.5,2.6443695781609473);
		\draw [line width=1.4pt,color=aqaqaq] (0.,-0.5360992384404077) -- (0.,2.6443695781609473);
		\draw [samples=50,domain=-0.99:0.99,rotate around={90.:(0.,0.)},xshift=0.cm,yshift=0.cm,line width=1.4pt] plot ({1.*(1+(\x)^2)/(1-(\x)^2)},{1.*2*(\x)/(1-(\x)^2)});
		\draw [samples=50,domain=-0.99:0.99,rotate around={90.:(0.,0.)},xshift=0.cm,yshift=0.cm,line width=1.4pt] plot ({1.*(-1-(\x)^2)/(1-(\x)^2)},{1.*(-2)*(\x)/(1-(\x)^2)});
		\draw [line width=1.4pt,color=qqqqff] (-1.2698132222214689,1.6162999781378669)-- (1.269813222221468,1.6162999781378662);
		\draw [samples=50,domain=-0.99:0.99,rotate around={0.:(0.,0.6186970324358333)},xshift=0.cm,yshift=0.6186970324358333cm,line width=1.4pt,color=qqzzff] plot ({0.7856296723362058*(1+(\x)^2)/(1-(\x)^2)},{0.7856296723362058*2*(\x)/(1-(\x)^2)});
		\draw [samples=50,domain=-0.99:0.99,rotate around={0.:(0.,0.6186970324358333)},xshift=0.cm,yshift=0.6186970324358333cm,line width=1.4pt,color=qqzzff] plot ({0.7856296723362058*(-1-(\x)^2)/(1-(\x)^2)},{0.7856296723362058*(-2)*(\x)/(1-(\x)^2)});
		\draw [line width=1.4pt,color=qqzzff] (0.,0.6186970324358333) -- (1.269813222221468,1.6162999781378662);
		\draw [line width=1.4pt,color=yqqqqq] (0.,0.3464867559163982) -- (1.269813222221468,1.6162999781378662);
		\draw [line width=1.4pt,color=qqqqff] (0.,0.) -- (1.269813222221468,1.6162999781378662);
		\begin{normalsize}
			\draw [fill=black] (0.,0.) circle (1.5pt) node[left] {$\odot \unip$};
			\draw [fill=qqzzff] (0.,0.6186970324358333) circle (1.5pt) node[left] {$\odot S(b)$};
			\draw [fill=qqqqff] (1.269813222221468,1.6162999781378662) circle (1.5pt) node[below right] {$N(f)=\isof(f)$};
			\draw [fill=qqqqff] (0.,1.6162999781378662)  circle (1.5pt) node[above left] {$\isomb(b)$};
			\draw [fill=yqqqqq] (0.,0.3464867559163982) circle (1.5pt) node[left] {$\ccmb(b)$};
		\end{normalsize}
	\end{tikzpicture}
	\caption{Local configuration in a Koebe isothermic net $\iso$ and the corresponding null congruence $\cc$. We see the unit-sphere $\unip$ (black), the circle $\iso(b)$ and the normal vector $N(f) = \isof(f)$ (dark blue), the sphere $S(b)$ and its center $\odot S(b)$ (teal), and finally the center of the null-sphere $\ccb(b)$ as well as its generator containing $\isof(f)$ (red).}
	\label{fig:weierstrass}
\end{figure}

However, the main advantage of the Weierstraß representation is to have explicit formulas. Therefore, we now briefly explain how to obtain explicit formulas not just for $\cpw^*$ but also for $\cpb^*$.

First, note that since $\iso$ is a Koebe isothermic net there is also a sphere $S(b)$  
that intersects $\unip$ orthogonally in the circle $\isob(b)$, see Figure~\ref{fig:weierstrass}.
The goal is to calculate $\ccmb^*(b) - \isof^*(f)$ for incident $b\in \Zb^2$, $f\in F(\Z^2)$, which we do by using Equation~\eqref{eq:nullspheredual} which shows that
\begin{align}
	\ccmb^*(b) - \isof^*(f) = \frac{\ccmb(b) - \isof(f)}{r_\bullet(b)^2}.
\end{align}
Here, $r_\bullet(b)$, the radius of the circle $\isob(b)$, may be calculated by three instances of Pythagoras' theorem in Figure~\ref{fig:weierstrass}. This yields
\begin{align}
	r_\bullet(b) =\frac{R_\bullet(b)}{\sqrt{1-R_\bullet(b)^2}},
\end{align}
where $R_\bullet(b)$ is the radius of $S(b)$
which is expressed, analogously to Equation~\eqref{eq:weierstrassradius}, as
\begin{align}
	R_\bullet(w) = \frac{2 \rho_\bullet(w)}{1 - |{\Cpmb(w)}|^2 + \rho_\bullet(w)^2}.\label{eq:dualiso}
\end{align}
Thus, it remains to calculate $\ccmb(b) - \isof(f)$.  The point of contact $\isof(f)$ an \emph{edge normal} of both $\iso$ and $\iso^*$, which is why we denote $\isof(f)$ by $N(f)$ for the following computations. The normal $N(f)$ is obtained by stereographic projection, that is
\begin{align}
	N(f) = \mathrm{Re}\left[ \frac{1}{1-|\Cpf(f)|^2} \begin{pmatrix}
		\phantom{-i}2 \Cpf(f) \\ -i2 \Cpf(f)  \\ 1 + |\Cpf(f)|^2 
	\end{pmatrix}\right].
\end{align} 
Additionally, let $b'$ be the other black vertex adjacent to $f$. We can write
\begin{align}
	\odot S(b) - N(f) = R_\bullet(b)T_\bullet(f), 
\end{align}
 with
\begin{align}
		T_\bullet(f) = \mathrm{Re}\left[\frac{1}{1-|\Cpf(f)|^2}\frac{\overline{\Cpmb(b)} - \overline{\Cpmb(b')}}{|\Cpmb(b) - \Cpmb(b')|} 
	\begin{pmatrix}
		\phantom{i(}1 + |\Cpf(f)|^2\phantom{)} \\ i(1 - |\Cpf(f)|^2) \\ 2 \Cpf(f)
	\end{pmatrix}\right], \label{eq:tblack}
\end{align}
analogously to Equation~\eqref{eq:weierstrass}. Moreover, a small calculation shows that
\begin{align}
	\ccmb(b) = \frac{1}{1\pm R_\bullet(b)}{\odot S}(b) = \frac{1}{1\pm R_\bullet(b)}(N(f) + R_\bullet(b) T_\bullet(f)).
\end{align}
Thus, we are able to write
\begin{align}
	\ccmb(b) - \isof(f) = \frac{R_\bullet(b)}{1\pm R_\bullet(b)}(T_\bullet(f) \mp N(f)),
\end{align}
and therefore, using Equation~\eqref{eq:dualiso},
\begin{align}
	\ccmb^*(b) - \isof^*(f) = \frac{1\mp R_\bullet(b)}{R_\bullet(b)}(T_\bullet(f) \mp N(f)). \label{eq:weiercontactb}
\end{align}
Furthermore, we have that, where $T_\circ(f)$ is obtained analogously to Equation~\eqref{eq:tblack} as
\begin{align}
	\isomw^*(w) - \isof^*(f) = \pm R_\circ^*(w)  T_\circ(f). \label{eq:weiercontact}
\end{align}

In clonclusion, we are able to obtain combined integral formulas for
\begin{enumerate}
	\item $\ccmw^*$, $\isof$ and $\ccmb^*$ (using Equation~\eqref{eq:weiercontactb} and Equation~\eqref{eq:weiercontact}), 
	\item all of $\ccw^*$ and $\ccb^*$ (as we have formulas for the radii of these spheres),
	\item $\cpmw^*$ and $\cpmb^*$ (via projection of our formulas to the first to coordinates), 
	\item the incircular net $\cp^*$ in terms of the hyperbolic orthogonal circle pattern $\Cp$.
\end{enumerate}


\section{The associated family} \label{sec:associated}

\begin{figure}
	\centering
	\begin{overpic}[scale=0.17]{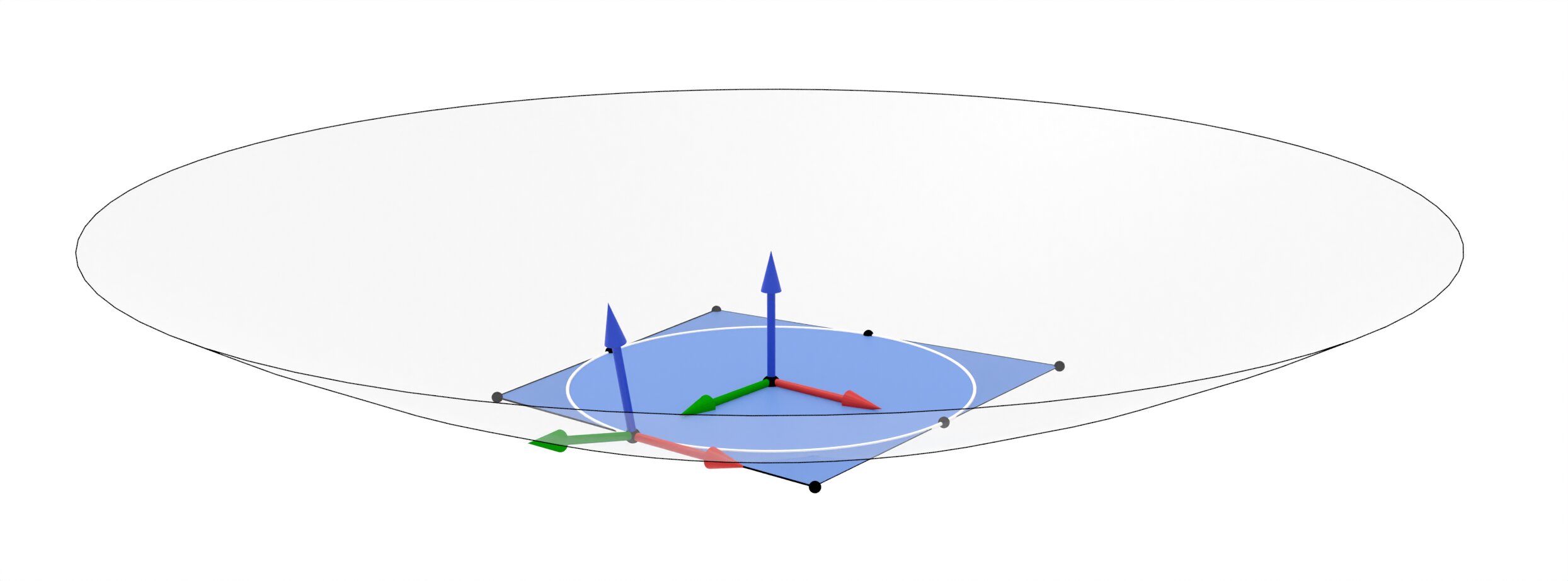}		
		\put(50,13){$m$}
		
		\put(50,19){$n$}
		\put(54.5,12){$T_1$}
		\put(44,12.24){$b_1$}
		
		\put(39.2,17){$N_1$}
		\put(45,5){$T_1$}
		\put(33,6.2){$B_1$}
		
		\put(39,6.7){$K_1$}
		\put(60.5,8){$K_2$}
		
		\put(28,12){$P_1$}
		\put(52,3.5){$P_2$}
		\put(68.5,13){$P_3$}

		\put(10,12){$\unip$}
	\end{overpic}
	\caption{The frames $(T_i,n,b_i)$ and $(T_i, N_i, B_i)$ used in the calculations for the associated family.}
	
	\label{fig:frames}
\end{figure}
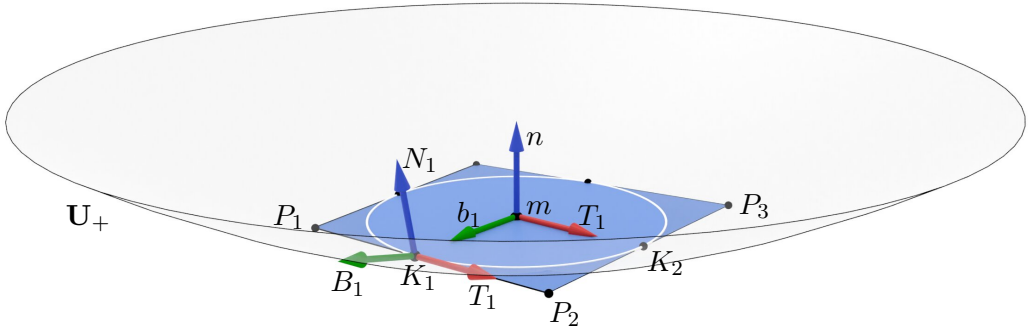

As discussed in Section~\ref{sec:smMaximal}, a smooth maximal surface $f$ comes with an associated family $(f^\varphi)_{\varphi \in \S^1}$ of maximal surfaces. Each member $f^\varphi$ of the associated family is obtained by rotating $\d f$ by $\varphi$ around the surface normal and then integrating the rotated $\d f$. In \cite{bhsminimal} the authors introduce an analogous discrete construction for minimal S-isothermic surfaces in Euclidean space. In this section we study how the discrete construction translates to maximal congruences in Lorentz space and the corresponding incircular nets in $\eucl$.

Let $\cc$ be a Koebe congruence and $\iso$ the corresponding S-isothermic net, and let $\isoc$ be the combination of $\iso$ and $\isof$ as in Equation~\eqref{eq:combinediso}. Fix an angle $\varphi \in \S^1$. For a face $f$, let $J_\varphi(f)$ be the (elliptic) Lorentz rotation with angle $\varphi$ around the (timelike) axis spanned by the center of $\unip$ and $\isof(f)$. Define the associated dual differential via
\begin{align}
	\d \isomw^\varphi(w,w') = \pm (R^{-1}(w) + R^{-1}(w')) \frac{J_\varphi(f)[\d \isomw(w,w')]}{|\d \isomw(w,w')|},
\end{align}
for all white vertices $w,w'$ adjacent to the common face $f$.

Let us investigate the properties of $\d \isomw^\varphi$ with some calculations. To simplify notation, let $P = (P_1, P_2, P_3,P_4)$ be a planar quad tangential to the space-like unit-sphere $\unip$ (like the quads in $\isomw$ of a Koebe congruence). Let $T_1, T_2, T_3, T_4$ be the unit edge vectors (or unit tangent vectors) of the quad $P$. Also let $R_i$ be the radii of the spheres at the vertices, so that
\begin{align}
	P_{i+1} - P_i = (R_{i+1} + R_i) T_i.
\end{align}
Let $K_i$ be the touching point on the edge $P_iP_{i+1}$, so that $K_i - P_i = R_i T_i$. Also let $n$ be the normal vector of the incircle $C$ with center $m$ and radius $r$, which passes through the four points $K_i$. Let $b_i = T_i \times n$ be the (bi)normal vector to $T_i$ in the quad plane, so that $m - K_i = r b_i$. 

Recall that the dual quad $P^*$ is defined by
\begin{align}
	P^*_{i+1} - P^*_i = \pm (R_{i+1}^{-1} + R_i^{-1}) T_i =(R^*_{i+1} + R^*_i) T^*_i, \label{eq:dualquad}
\end{align}
where $R^*_i = R^{-1}_i$,  $T^*_i = \pm T_i$ and the sign is $+$ for horizontal edges and $-$ for vertical edges. The dual contact points satisfy $K^*_i - P_i^* = R^*_iT_i^*$, and the dual incircle radius is $r^* = r^{-1}$. Note that if $\gamma_i$ is the corner angle at $P_i$, then in the dual quad we have $\gamma_i^* = \pi - \gamma_i$.

In order to obtain the associated family, let us also consider at each point $K_i$ the normal vector to $\unip$, which we denote by $N_i$. Note that $K_i$ coincides with $N_i$ by construction, but this will not be the case in the associated family so we use different symbols. To express $N_i$ and $B_i := T_i \times N_i$, we use the rotation matrix in the orthonormal basis $(T_i, n, b_i)$ 
\begin{align}
	M_\theta = \begin{pmatrix}
		1 & 0 & 0 \\
		0 & \cosh\theta & \sinh \theta \\
		0 & \sinh \theta & \cosh \theta
	\end{pmatrix},
\end{align}
so that
\begin{align}
	T_i &= M_\theta T_i = T_i,\\
	N_i &= M_\theta n = n \cosh \theta + b_i \sinh \theta, \\
	B_i &= M_\theta b_i = n \sinh \theta + b_i \cosh \theta.
\end{align}
Following \cite{bhsminimal}, rotation in the associated family means rotating $T_i$ around $N_i$ by an angle $\varphi$. In the orthonormal basis $(T_i, N_i, B_i)$ the corresponding rotation matrix is
\begin{align}
	M_\varphi = \begin{pmatrix}
			\cos \varphi & 0 & -\sin \varphi\\
			0 & 1 & 0 \\
			\sin \varphi & 0 & \phantom{-}\cos \varphi
		\end{pmatrix}, \label{eq:mphi}
\end{align}
so that
\begin{align}
	&T_i^\varphi & &= M_\varphi T_i & & = T_i \cos \varphi + B_i \sin \varphi  & &= T_i \cos \varphi + n \sinh \theta \sin \varphi + b_i \cosh \theta \sin \varphi, & &\label{eq:associatedtangent}\\
	&N_i^\varphi & &= M_\varphi N_i & &= N_i  & &=  n \cosh \theta + b_i \sinh \theta, & & \label{eq:associatedtangent1}\\
	&B_i^\varphi & &= M_\varphi B_i & &= -T_i \sin \varphi + B_i \cos \varphi & & =  -T_i \sin \varphi + n \sinh \theta \cos \varphi + b_i \cosh \theta \cos \varphi. & & \label{eq:associatedtangent2}
\end{align}
The associated quad $P^\varphi$ is obtained by integrating
\begin{align}
	P^\varphi_{i+1} - P^\varphi_i = \pm (R^*_{i+1} + R^*_i) T_i^\varphi. \label{eq:assointegral}
\end{align}
This formula closes in the $T_i$ components since it closes for $\pm T_i = T^*_i$  already, for the $n$ components it is just a telescopic sum, and for the $b_i$ component it follows from the $T_i$ components and the linearity of $b_i = T_i \times n$. Let us put this in a theorem.

\begin{theorem}
	The associated discrete differential $\d \isomw^\varphi$ is closed.
\end{theorem}
\begin{proof}
	It suffices to check the closing condition around each black vertex, which corresponds to the fact that $P^\varphi$ is well-defined in the calculation above.
\end{proof}

Therefore we may integrate $\d \isomw^\varphi$ to obtain a net $\isomw^\varphi$, which we consider to be a surface in the associated family of $\isomw^*$. As we will see in the next section, $\isomw^\varphi$ shares many properties with S-isothermic nets, but not all, and is therefore not an S-isothermic net. Instead, we consider $\isomw^\varphi$ to be a discrete maximal surface that is still conformally parametrized -- but \emph{not} along curvature lines. Thus, we think of $\isomw^\varphi$ as a discretization of an isothermic surface but not as a discretization of a curvature-line parametrization.

\begin{figure}	\centering
	\begin{minipage}{\linewidth}
		\includegraphics[width=.45\linewidth]{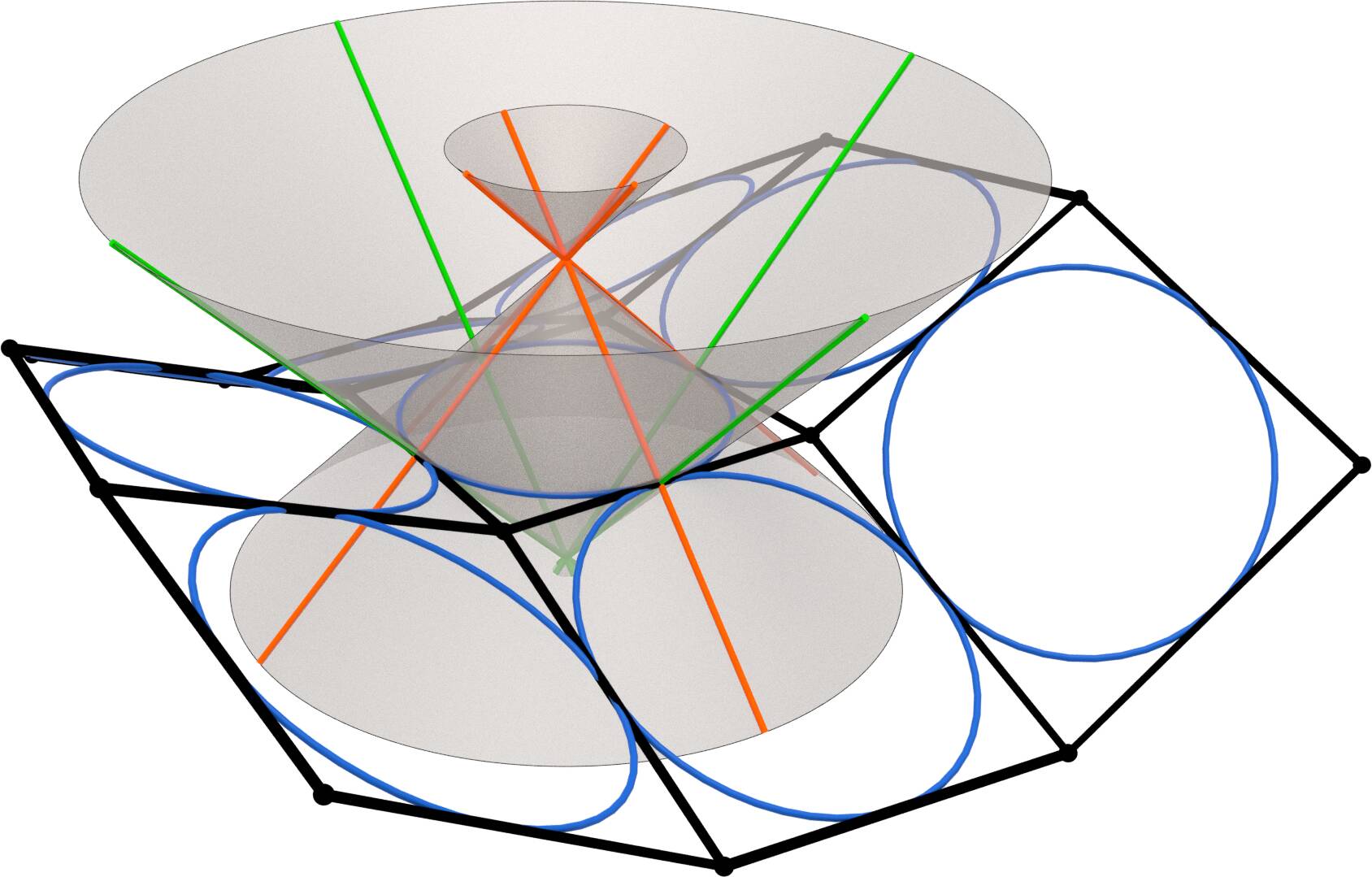}
		\hspace{1cm}
		\includegraphics[width=.45\linewidth]{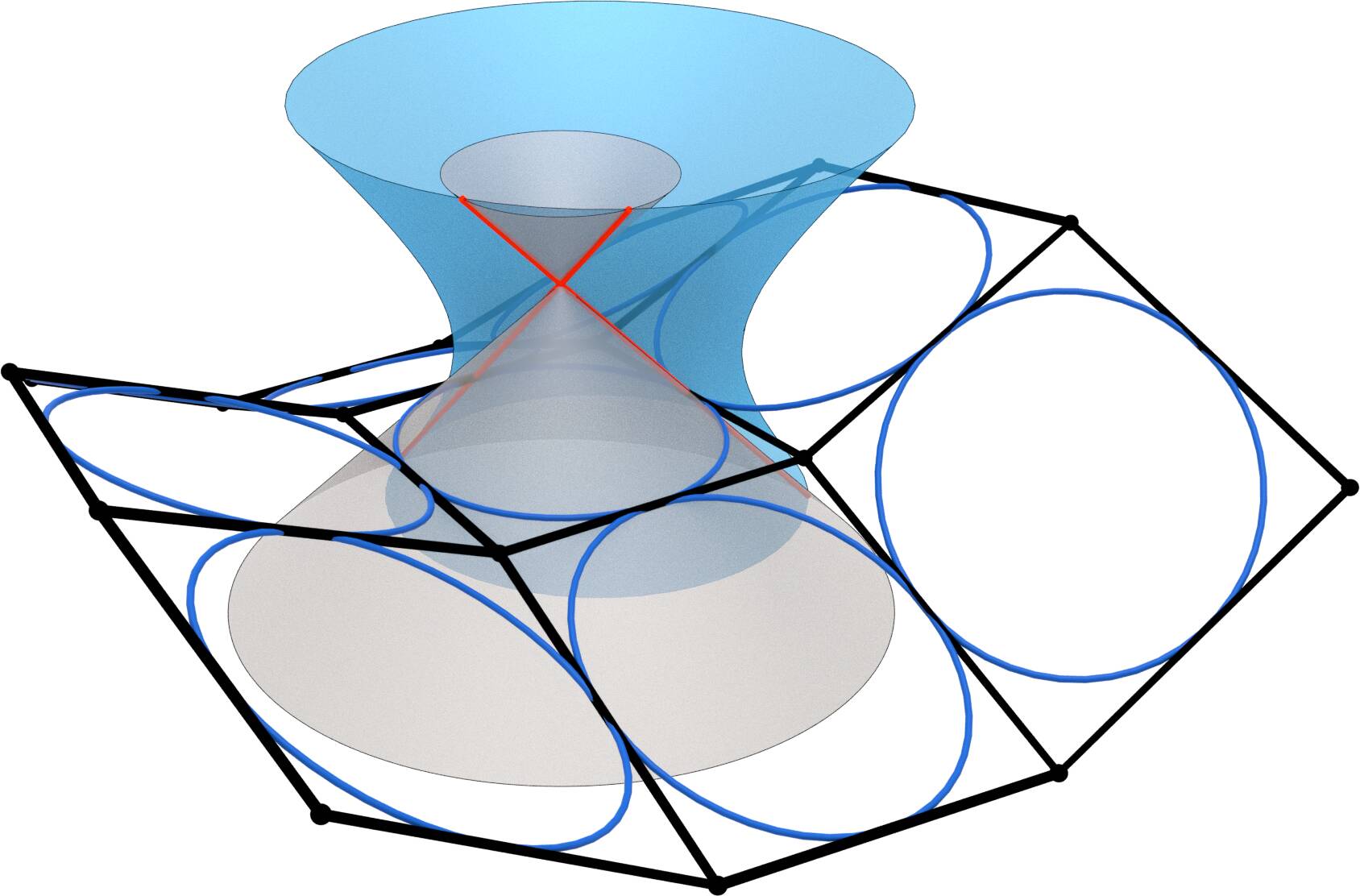}
		\vspace{.2cm}
		\newline
		\includegraphics[width=.45\linewidth]{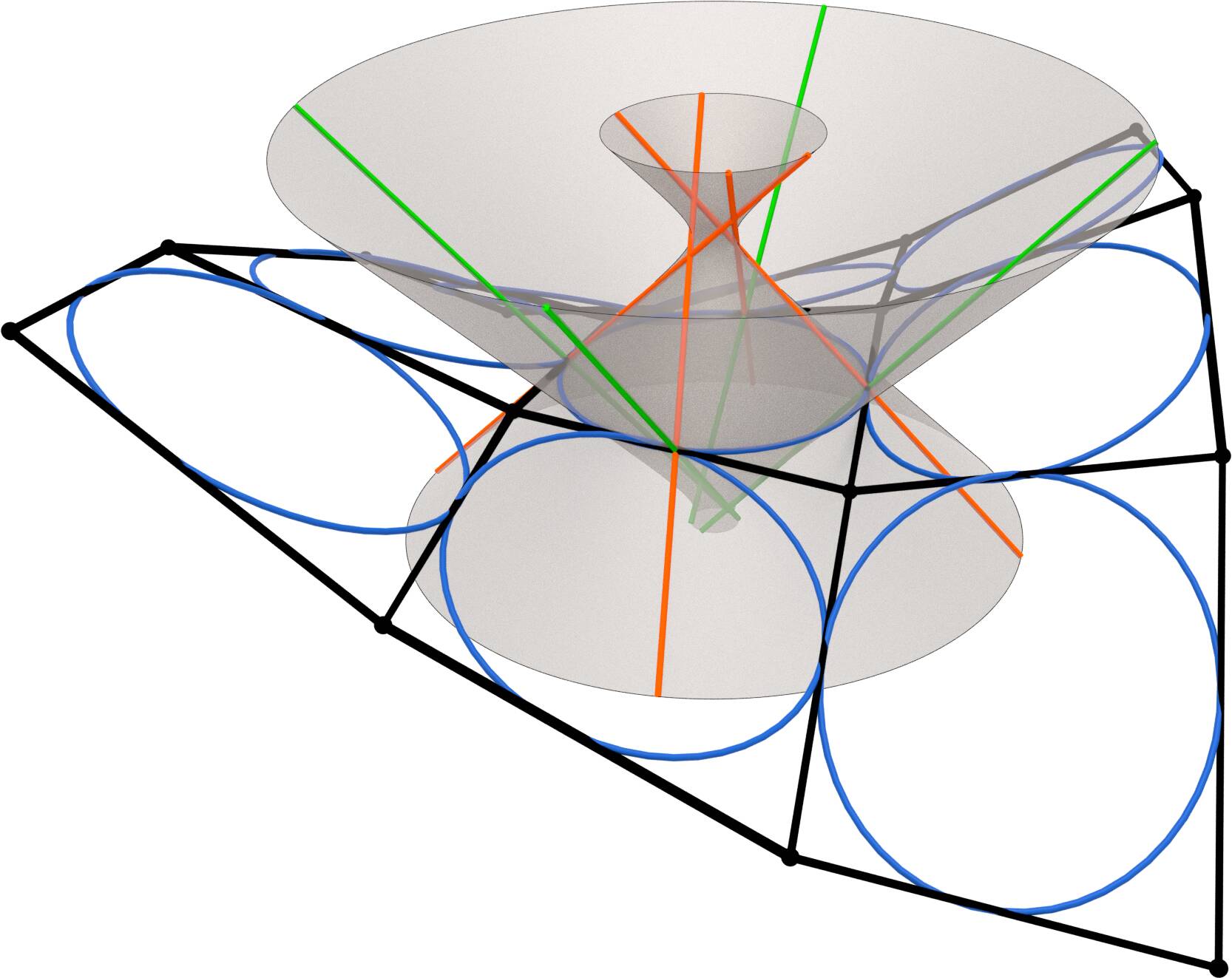}
		\hspace{1cm}
		\includegraphics[width=.45\linewidth]{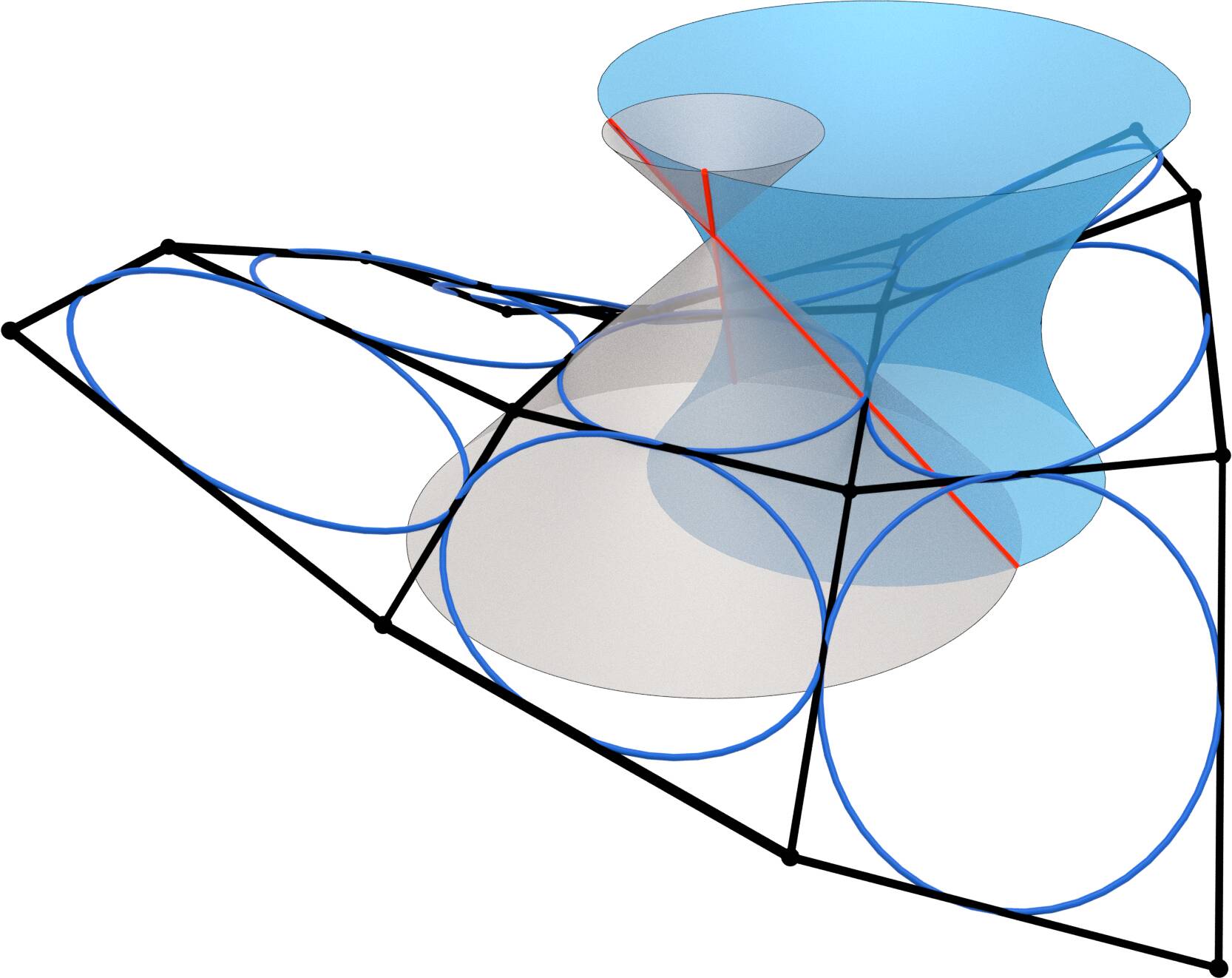}
		\vspace{.2cm}
		\newline
		\includegraphics[width=.45\linewidth]{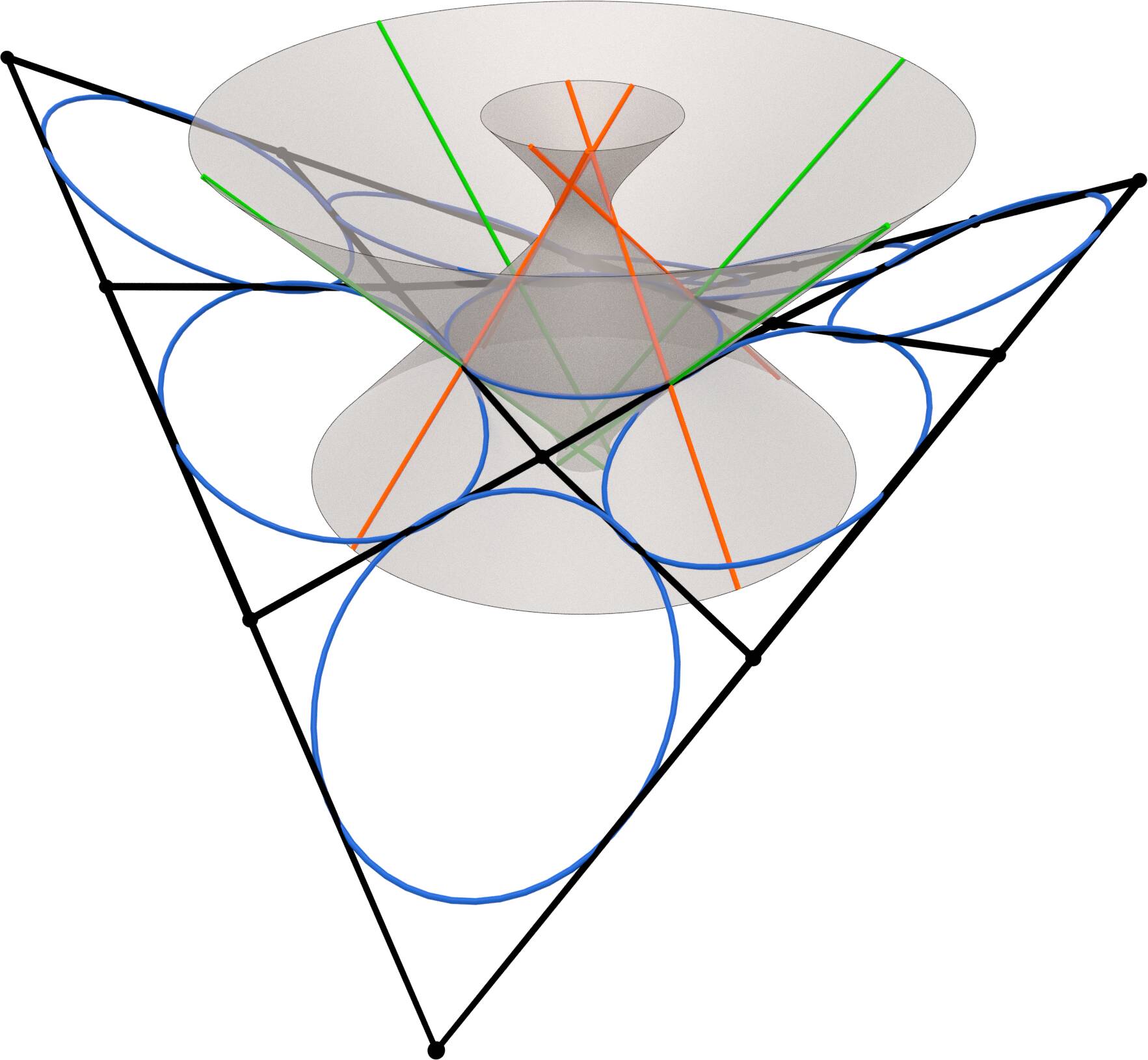}
		\hspace{1cm}
		\includegraphics[width=.45\linewidth]{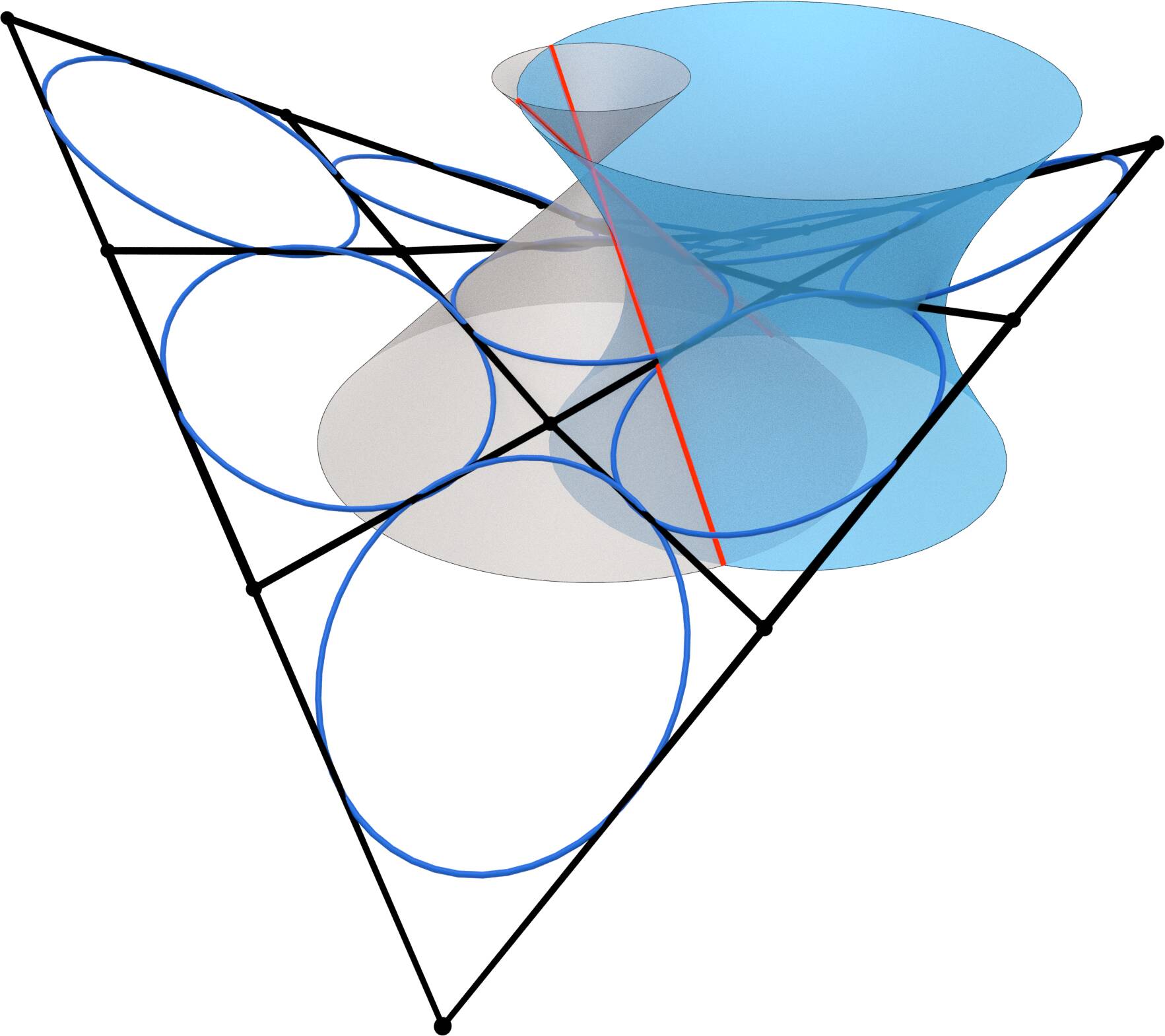}
		\vspace{.2cm}
		\newline
		\includegraphics[width=.45\linewidth]{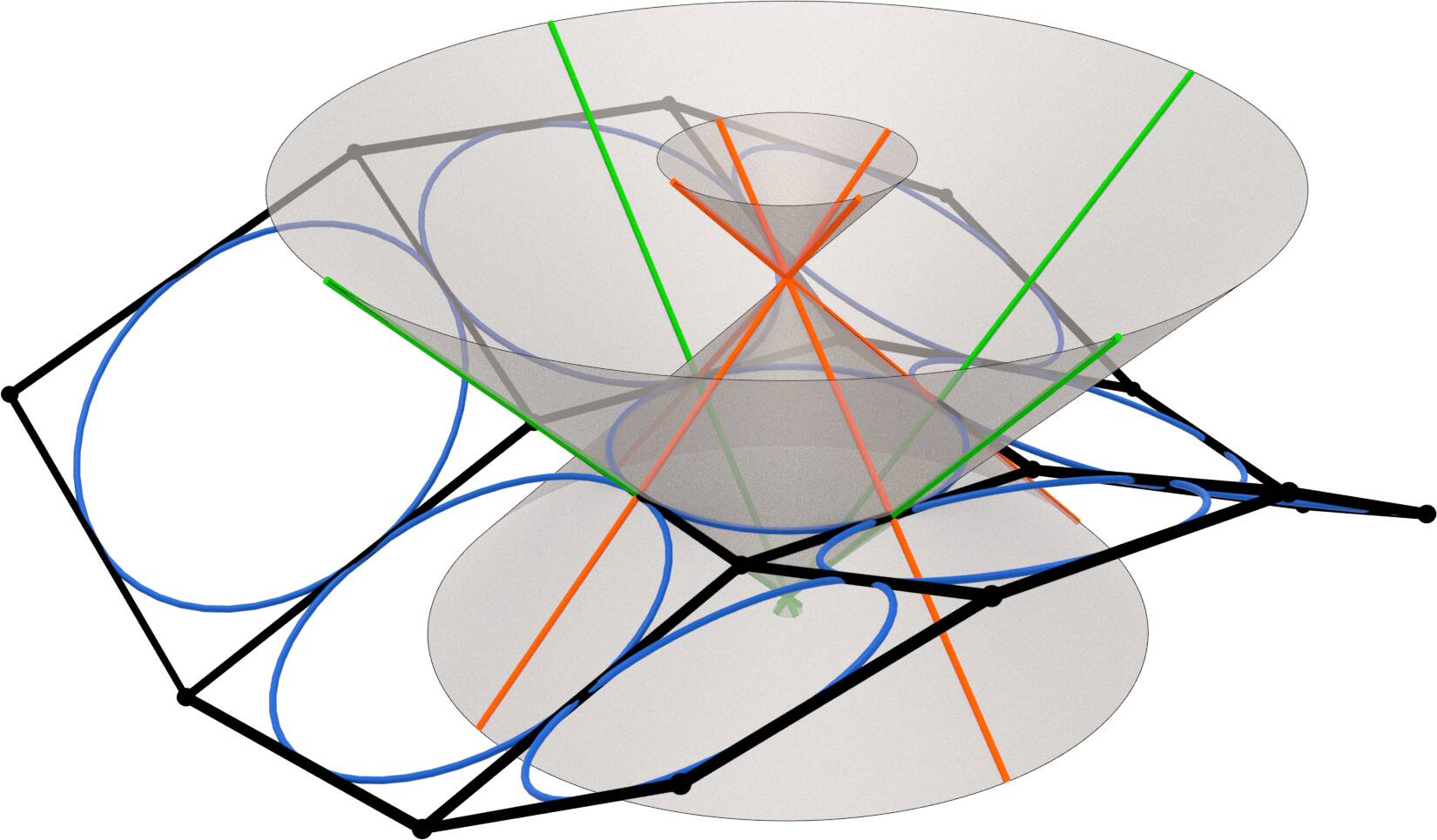}
		\hspace{1cm}
		\includegraphics[width=.45\linewidth]{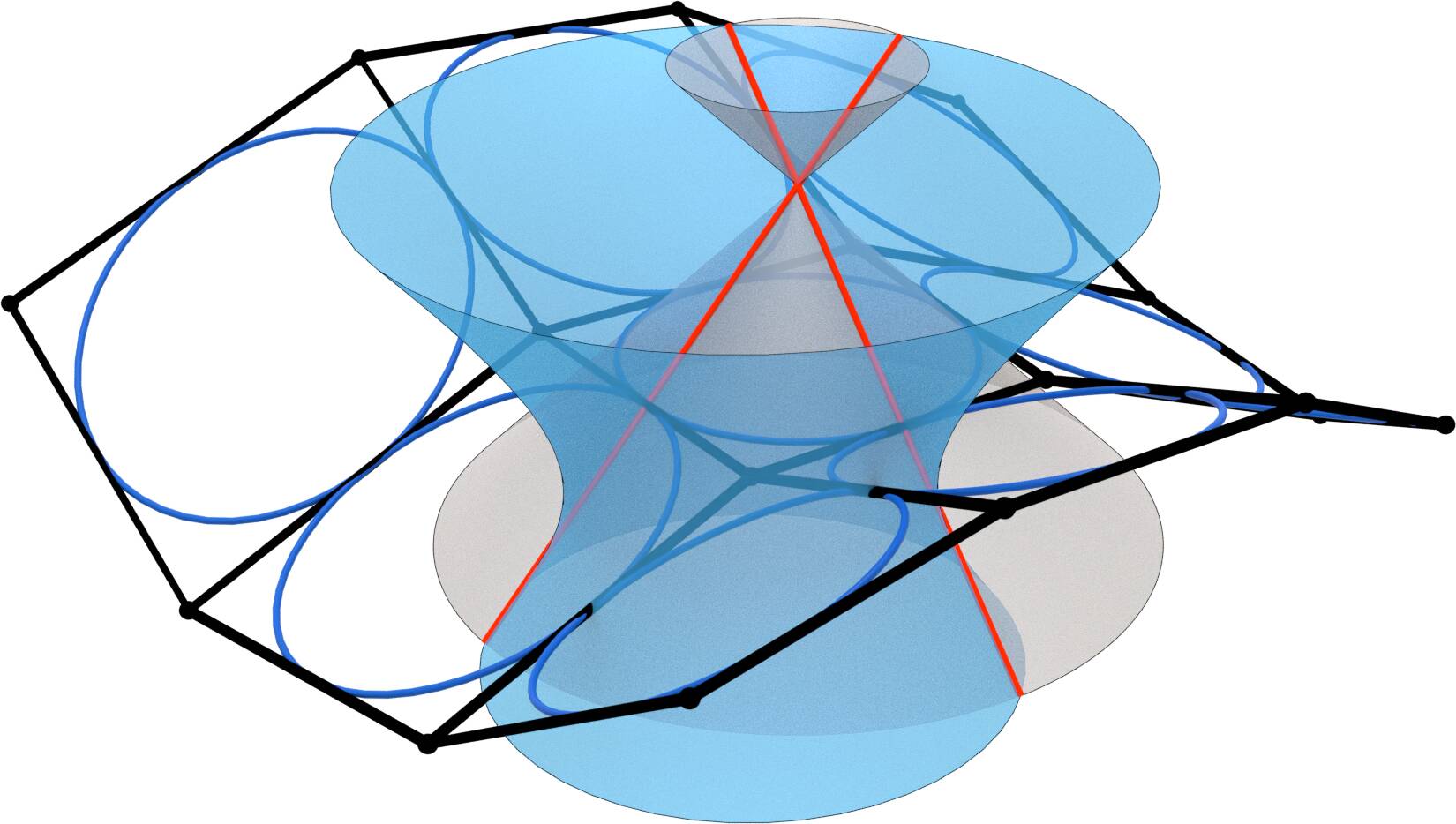}
	\end{minipage}
	\caption{Congruences in the associated family of a maximal surface, from top to bottom for the values $\varphi = 0, \frac{\pi}{4}, \frac{\pi}{2}, \pi$, the latter is a rotation of the original surface. Left: Spheres of the contact congruences $(\cc^1)^\varphi$ and $(\cc^2)^\varphi$ of radii $\sin \varphi$. The spheres contain adjacent generators and corresponding face circles. Right: A null sphere and a timelike vertex sphere of one of the associated null congruences $\dot \cc^\varphi$. }
	\label{fig:asso}
\end{figure}

\section{The associated congruences}  \label{sec:radii}

Before we begin, let us recall the notion of a contact congruences \cite{admpsiso} which generalizes null-congruences (Definition~\ref{def:nullcongruence}) and which we need in the following.
\begin{definition} \label{def:contactcongruence}
	A \emph{contact congruence} $c$ is a pair of maps
	\begin{align}
		\cc&: \Zw^2 \rightarrow \osp_-(\lor) \cup \sp_0(\lor), & 
		\ccf&: F(\Z^2) \rightarrow \oli_0(\lor),
	\end{align}
	such that the oriented sphere $\cc(v)$ is in oriented contact with the oriented isotropic line $\ccf(f)$ whenever $v$ and $f$ are incident.
\end{definition}

In the previous section, we followed the technique of \cite{bhsminimal} translated to Lorentz geometry -- with a few more details. In the following we take the calculations further to prove some new results (which do not readily translate from or to Euclidean geometry).

Since a quad $P_1^\varphi$, $P_2^\varphi$, $P_3^\varphi$, $P_4^\varphi$ is obtained by integrating as in Equation~\eqref{eq:assointegral}, there are still spheres $\isow^\varphi$ at the white vertices with radii given by $R_i^\varphi = R_i^* = R_i^{-1}$. Moreover, for adjacent white vertices the spheres are touching. Consequently, the touching coins lemma \cite{bhsminimal} implies that there is still a circle $C^\varphi$ which passes through the four points of contact $K_1^\varphi$, $K_2^\varphi$, $K_3^\varphi$, $K_4^\varphi$. However, circles of adjacent faces are not touching. 

Let us consider the orthogonal projection $Q^\varphi$ of the quad $P^\varphi$ to the plane of the circle $C^\varphi$. 

\begin{lemma}\label{lem:assosimilarity}
	The projected quad $Q^\varphi$ is (Lorentz) similar to $P^*$, and the scaling factor is $\mu = \sqrt{1 + r^2 \sin^2 \varphi}$.
\end{lemma}
\proof{
	Projecting the rotated vectors $P^\varphi$ onto the plane of $C^\varphi$ gives a scaled and rotated image of $P^*$. However, the scaling and rotation only depend on the angles $\varphi$ and $\theta = \angle(N_i,n)$. These angles are the same for all edges of $Q^\varphi$. Hence, the quad $Q^\varphi$ is similar to $P^*$. The scaling factor can be derived from the quotient of lengths of the orthogonal projection of $T^\varphi$ denoted by $U^\varphi$ and $T^\varphi$.
	We obtain from Equation~\eqref{eq:associatedtangent} that
	\begin{align}
		U^\varphi_i &= T_i \cos \varphi + b_i \cosh \theta \sin \varphi.
	\end{align}
	Consequently,
	\begin{align}
		\mu = |U^\varphi_i| = \sqrt{\cos^2 \varphi + \cosh^2\theta \sin^2\varphi} = \sqrt{1 + r^2 \sin^2 \varphi},
	\end{align}	
	is the scaling factor of the edge lengths (we used that $\sinh \theta = r$). \qed
}

As a consequence of Lemma~\ref{lem:assosimilarity}, the quad $Q^{\varphi}$ has an incircle $D^{\varphi}$ with radius $\mu r^*$, and center $m^\varphi$ which coincides with the center of $C^\varphi$.

As mentioned above, the spheres centered at $P_i^\varphi$ and $P^\varphi_{i+1}$ still touch in $K_i^\varphi$. Let us call the two common isotropic lines of these spheres $G^\varphi_{i,\pm}$.

\begin{lemma} \label{lem:assocc}
	The four isotropic lines $G^\varphi_{1,+}$, $G^\varphi_{2,-}$, $G^\varphi_{3,+}$, $G^\varphi_{4,-}$ are contained in a sphere of radius $\rho_\varphi := \sin \varphi$. The four isotropic lines $G^\varphi_{1,-}$, $G^\varphi_{2,+}$, $G^\varphi_{3,-}$, $G^\varphi_{4,+}$ are also contained in a (generically different) sphere of radius $\rho_\varphi$, see Figure \ref{fig:asso} (left).
\end{lemma}

\proof{
	Note that $G^\varphi_{i,\pm}$ are both orthogonal to $T_i^\varphi$. Thus, the direction vectors of the two isotropic lines are given by
	\begin{align}
		N^\varphi_i \pm B^\varphi_i = \mp T_i \sin \varphi + n (\cosh \theta \pm \sinh \theta \cos \varphi) + b_i (\sinh \theta \pm \cosh \theta \cos \varphi).
	\end{align}
	The projection of the line $G^\varphi_{i,\pm}$ to the plane of the circle $C^\varphi$ is a line that has distance $\rho^\varphi$ to the center $m^\varphi$, where 
	\begin{align}
		\rho^\varphi &= \mu r^* \left< \underset{ \text{normal to } U_i^\varphi }{\underbrace{\frac{-T_i \cosh \theta \sin \varphi + b_i \cos \varphi}{\mu}}}
		,
		\underset{\text{normal of projected } G_{i,\pm}^\varphi}{
			\underbrace{
				\frac{T_i(-\sinh \theta \mp \cosh \theta \cos \varphi) \mp b_i \sin\varphi}{\sqrt{\sin^2 \varphi + (\sinh \theta \pm \cosh \theta \cos \varphi)^2}}
			}
		} \right> \\ 
		&= \frac{r^* \sinh \theta (\cosh \theta \pm \sinh \theta \cos \varphi) \sin \varphi}
		{\sqrt{ (\cosh \theta \pm \sinh \theta \cos \varphi)^2 }}  = r^* \sinh \theta \sin \varphi = \sin \varphi,
	\end{align}
	(independently of $i$), since $\sinh \theta = r$ and $rr^* = 1$. Because the lines $G^\varphi_{i,\pm}$ are isotropic, $\rho^\varphi$ is also the distance of $G^\varphi_{i,\pm}$ to the axis of $C^\varphi$. As the lines are also the same distance of the center of $C^\varphi$, there are two unique Lorentz spheres of radii $\rho^\varphi$ containing the quadruples of isotropic lines of the lemma. The circle $C^\varphi$ is contained in both spheres. Note that this is not the smallest Euclidean circle of the spheres. The change of sign of the generators in one sphere is due to the change of orientation of vertical edges when dualizing.\qed
}

Surprisingly, $\rho^\varphi$ does not depend on $r$ or (equivalently) on $\theta$. Hence, the radii $\rho^\varphi$ are the same for all quads in the associated surface. 

By combining the generalization of null congruences to contact congruences (Definition~\ref{def:contactcongruence}) with Lemma~\ref{lem:assocc}, we obtain the following theorem.

\begin{figure}
	
	\centering
	\begin{minipage}{\linewidth}
		
	\vspace{-.9cm}
	\includegraphics[width=.45\linewidth]{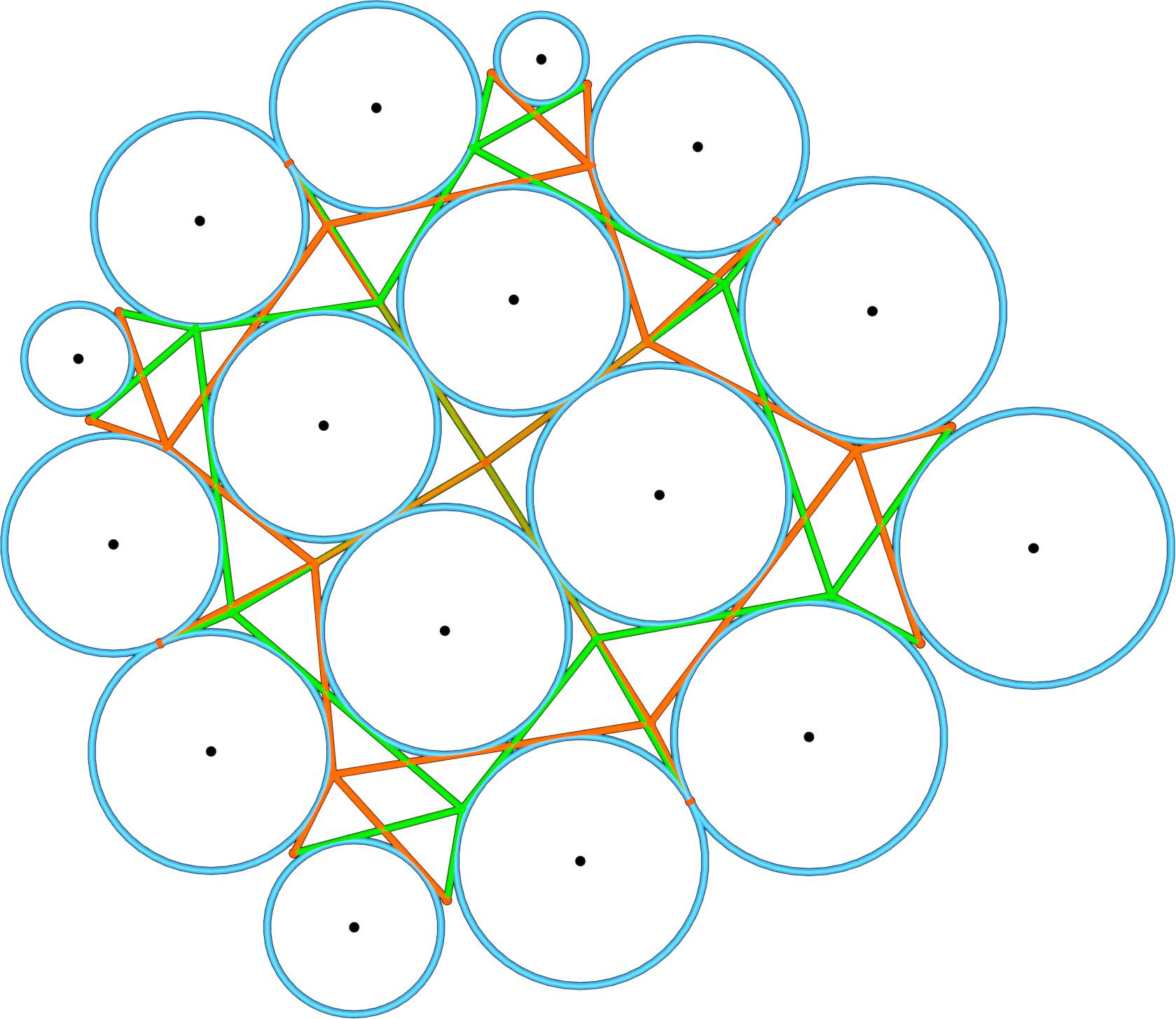}
	\hspace{1cm}
	\includegraphics[width=.45\linewidth]{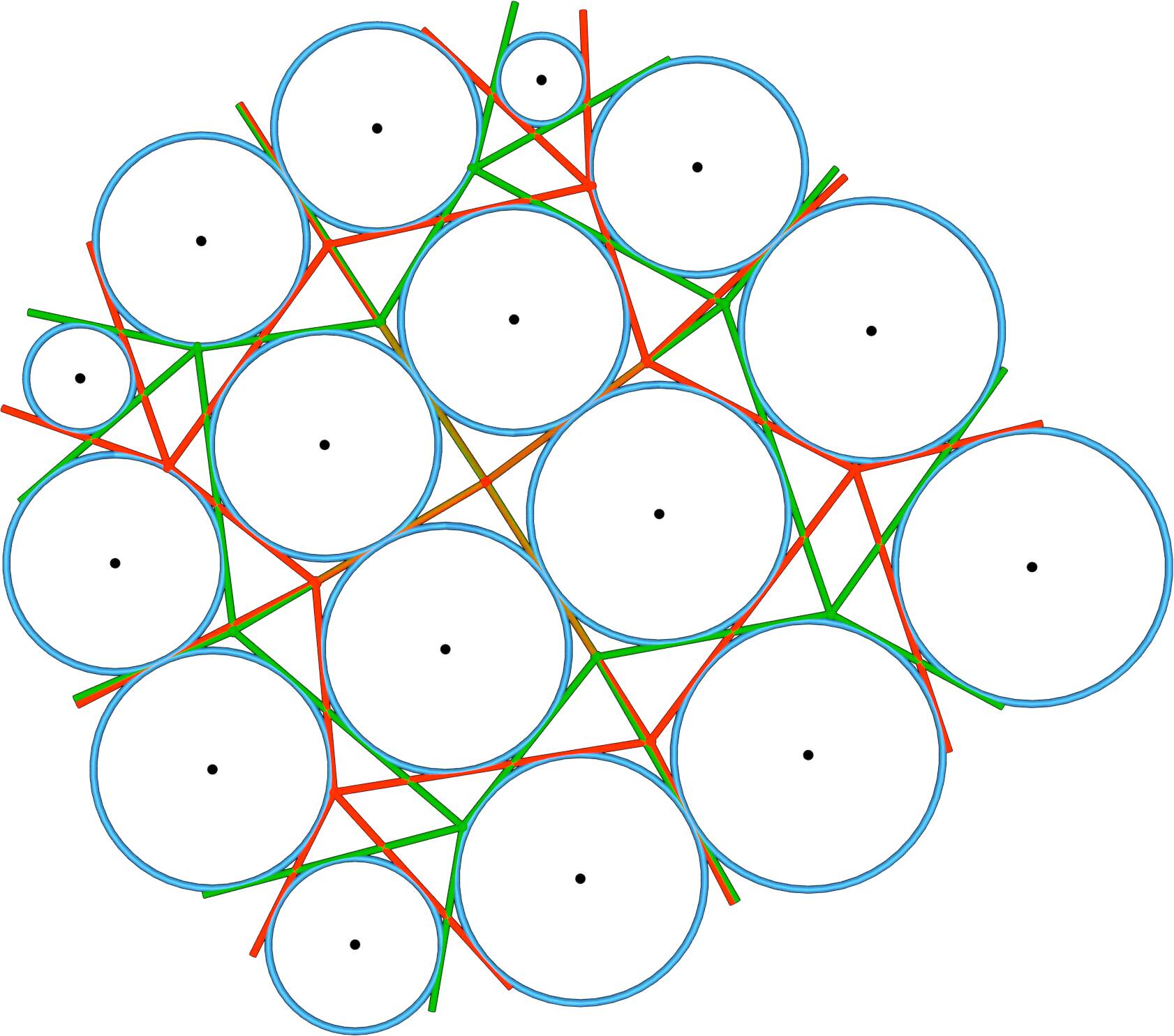}
	\vspace{.2cm}
	\newline
	\includegraphics[width=.45\linewidth]{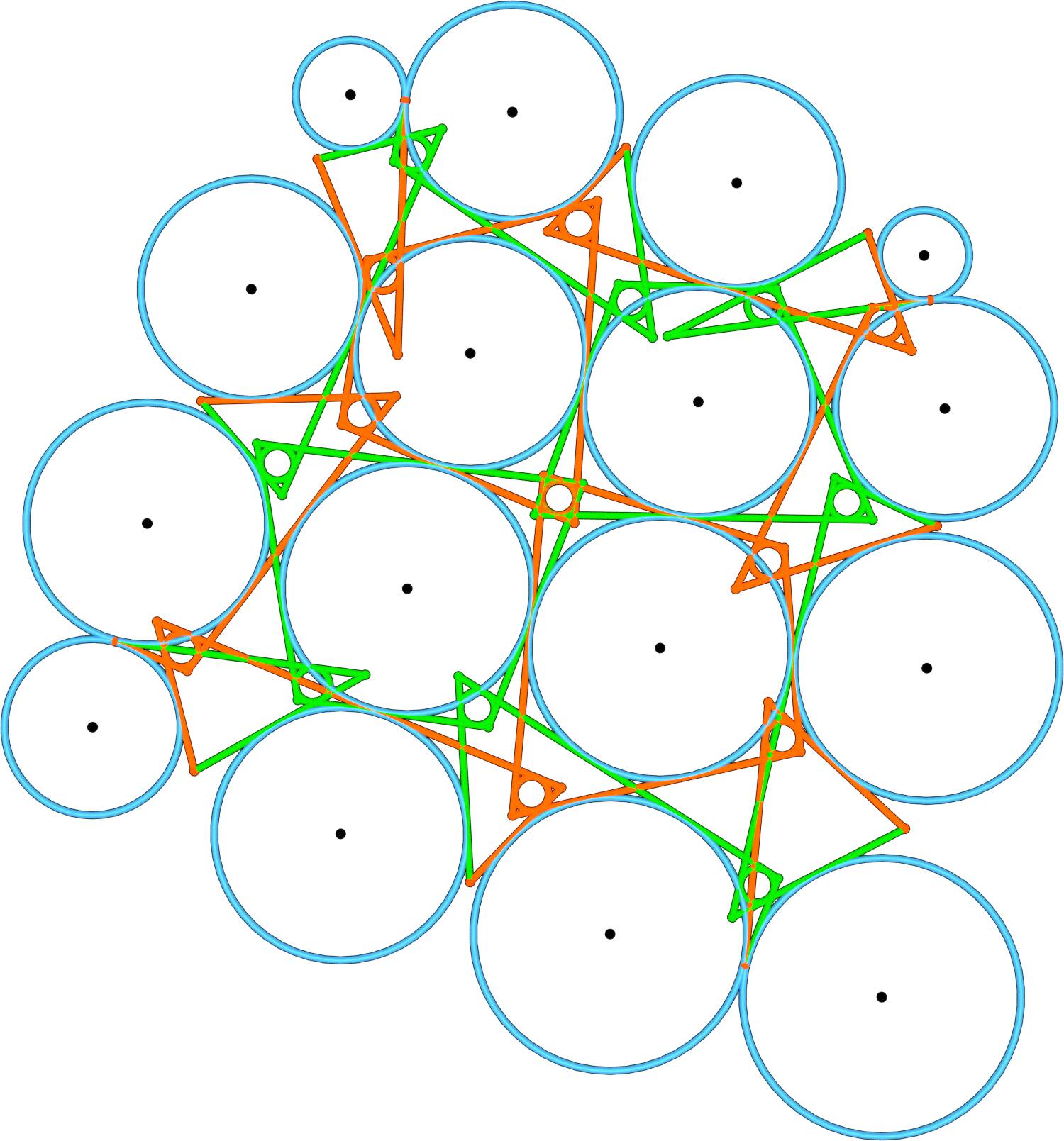}
	\hspace{1cm}
	\includegraphics[width=.45\linewidth]{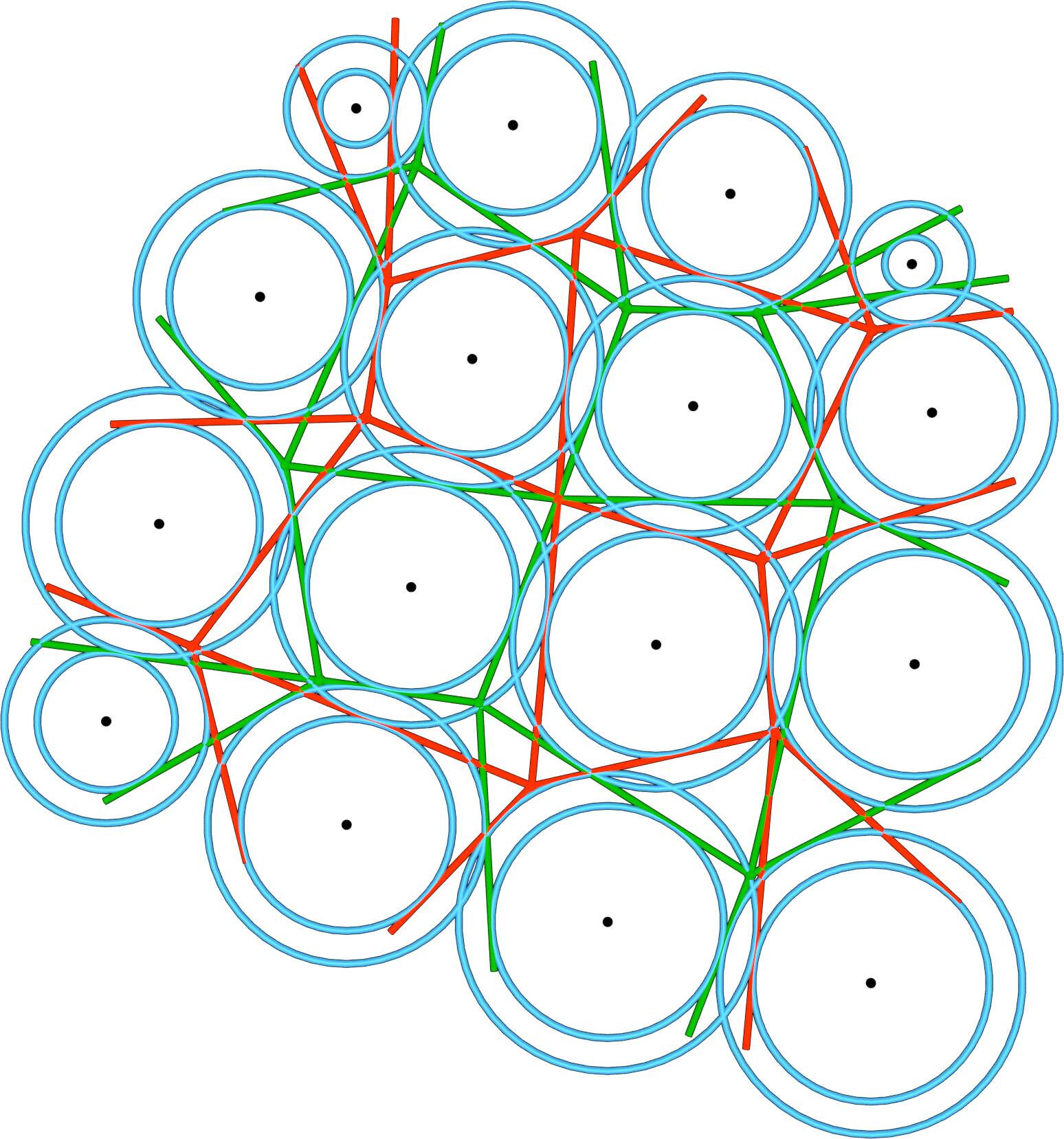}
	\vspace{.2cm}
	\newline
	\includegraphics[width=.45\linewidth]{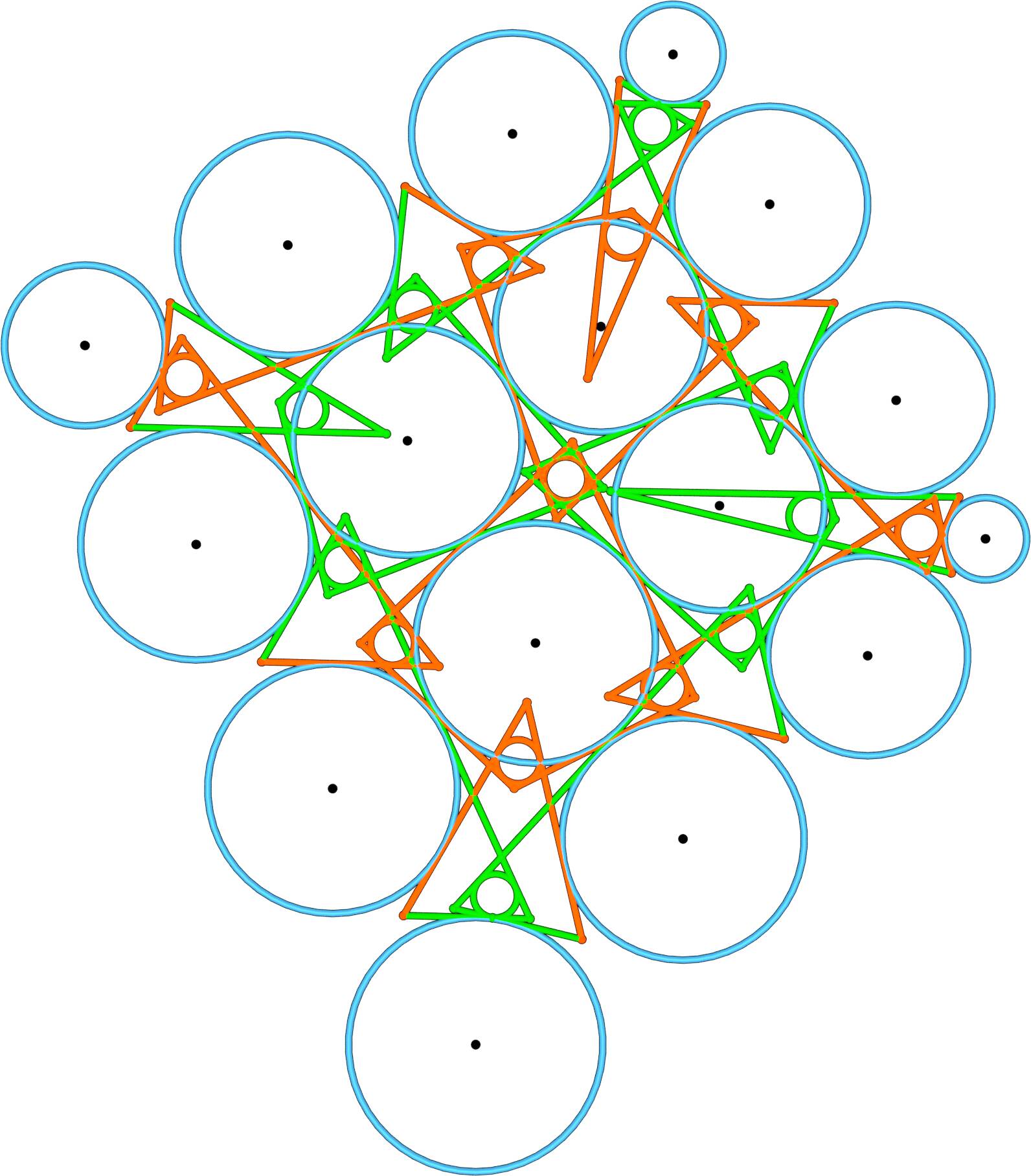}
	\hspace{1cm}
	\includegraphics[width=.45\linewidth]{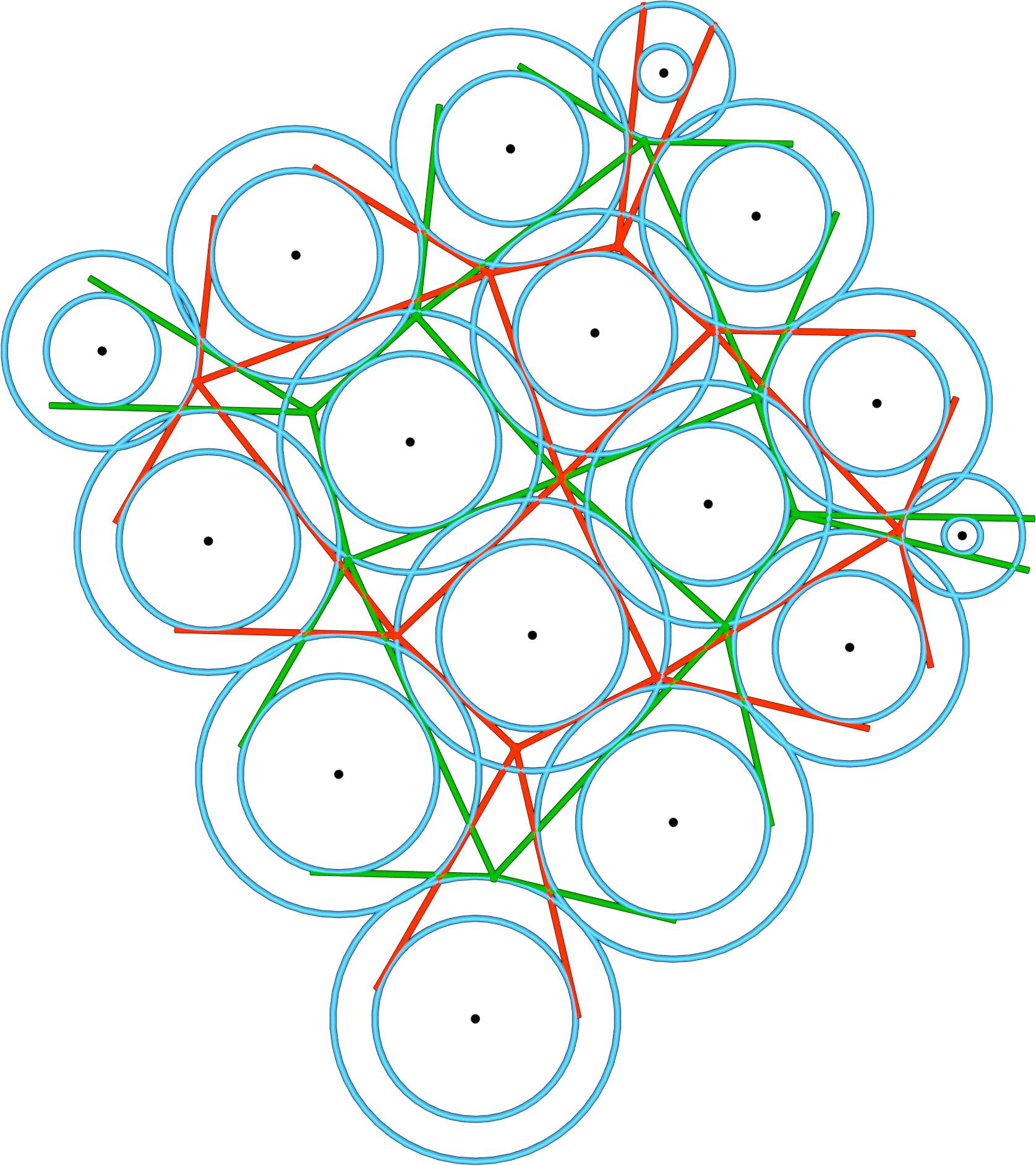}
	\caption{Planar nets obtained from the associated congruences in the associated family of a maximal surface (see Figure \ref{fig:asso}) for the values $\varphi = 0, \frac{\pi}{4}, \frac{\pi}{2}$, from top to bottom. Left: Projections of the isotropic lines and timelike spheres in the associated contact congruences $(\cc^1)^\varphi$ (green) and $(\cc^2)^\varphi$ (orange). All the green and orange circles have radius $\sin \varphi$. Right: The incircular nets obtained as projections of the associated null congruences $(\dot \cc^1)^\varphi$ (green) and $(\dot \cc^2)^\varphi$ (red).}
	\label{fig:icnets}
	\end{minipage}
\end{figure}

\begin{theorem}\label{th:associatedcc}
	Let $\iso$ be a Koebe net and $\iso^\varphi$ a member of the associated family of $\iso^*$ (the Christoffel dual S-maximal surface). 
	Then there are two contact congruences $(\cc^1)^\varphi$, $(\cc^2)^\varphi$ such that $(\ccw^1)^\varphi = (\ccw^2)^\varphi = \isow^\varphi$, which we call \emph{associated congruences}.
\end{theorem}
\begin{proof}
	By Lemma~\ref{lem:assocc}, for each black vertex $b$ there are exactly two spheres per face that are in oriented contact with the four spheres $\isow^\varphi(w_i)$ of adjacent white vertices $w_i$. Choosing one of the two spheres for some black vertex $b_0$ and then all other spheres consistently proves the claim. 
\end{proof}

Additionally, it is possible to define a new contact congruence $\dot\cc$ such that $\odot\dot\cc = \ccm$, and such that the radii of all spheres of $\dot\cc_\bullet$ are zero, while at white vertices the radii satisfy
\begin{align}
	\dot R(v) = R(v) - \rho_\varphi. \label{eq:radiusshift}
\end{align}

Here, $R$ and $\dot R$ denote the oriented radii of the oriented spheres, where the orientation of an oriented sphere is encoded in the sign of its radius. 
In this sense, the transformation in Equation~\eqref{eq:radiusshift} is just a Laguerre (offset-) transformation from the point of view of Laguerre geometry, see \cite{admpsiso}. From Equation~\eqref{eq:radiusshift} we obtain the following corollary.

\begin{corollary}\label{cor:assonull}
	Let $\cc^\varphi$ be an associated congruence. Then there is a null congruence $\dot \cc^\varphi$ with $\odot\dot\cc^\varphi = \ccm^\varphi$, see Figure~\ref{fig:asso} (right).
\end{corollary}

In particular, this means that for every member $\iso^\varphi$ of the associated family of a maximal surface there are two corresponding null congruences $(\dot \cc^1)^\varphi$, $(\dot \cc^2)^\varphi$.

\begin{remark} \label{rem:assocp}
	Corollary~\ref{cor:assonull} implies that every maximal incircular net $\cp$ also comes with an associated family of two incircular nets $(\cp^1)^\varphi$, $(\cp^2)^\varphi$ that are the projections of $(\dot \cc^1)^\varphi$ and $(\dot \cc^2)^\varphi$, see Figure \ref{fig:icnets} (right). These are not isothermic incircular nets. Instead, $(\cp^1)^\varphi$ has the property that around each black vertex $b$ the other tangents (as in Definition~\ref{def:isocp}) are in oriented contact with an oriented circle of radius $2\sin\varphi$ and center $(\cpm^2)^\varphi(b)$, and analogously for $(\cp^2)^\varphi$. It is not clear if this is a characterizing property for incircular nets in an associated family. This is in part because there currently is also no characterization of S-isothermic nets that are in an associated family. 
	The projections of the contact congruences $(\cc^1)^\varphi$ and $(\cc^2)^\varphi$ were introduced in \cite{admpsiso} as so called \emph{cycle patterns}  and are shown in Figure \ref{fig:icnets} (left).
\end{remark}

\begin{remark}\label{rem:miquel}
	In \cite{admpsiso} we also discussed Miquel dynamics applied to contact congruences using some Lorentz Lie sphere geometry. With the arguments developed there it is possible to obtain a calculation-free proof of Theorem~\ref{th:associatedcc} as follows: the four oriented timelike Lorentz spheres that are cyclically in contact span a 3-space of signature $\si{++--}$ in the Lie lift. Therefore the polar complement of this span is a line of signature $\si{+-}$ which intersects the Lie quadric twice. These two intersection points correspond to two oriented timelike spheres that are in oriented contact with the original four oriented timelike spheres. However, we did not obtain a proof of the constant radii without the calculations above. Moreover, it is not hard to see that the two contact congruences $\cc^1,\cc^2$ in Theorem~\ref{th:associatedcc} are actually related by one step of Miquel dynamics. Furthermore, due to Corollary~\ref{cor:assonull} it follows that iterated Miquel dynamics applied to associated congruences produces the sequence
	\begin{align}
		\dots \leftrightarrow \ccm^\varphi \leftrightarrow \ccm^\varphi \leftrightarrow \miq{\ccm^\varphi} \leftrightarrow \miq{\ccm^\varphi} \leftrightarrow \ccm^\varphi \leftrightarrow \ccm^\varphi \leftrightarrow \miq{\ccm^\varphi} \leftrightarrow \miq{\ccm^\varphi} \leftrightarrow \dots
	\end{align}
	This is the same behaviour as we showed in \cite{admpsiso} for isothermic congruences, except that for isothermic congruences it holds not just for the centers but also for the radii. It is currently unclear whether this type of Miquel sequence is characteristic for associated congruences. 
\end{remark}


\section{The \texorpdfstring{$X$}{X}-variables in the associated family}  \label{sec:xvariables} \label{sec:assox}

We begin by considering a vertex star of $\isow$ in contrast to the quad of $\isow$ we considered in Section~\ref{sec:associated}. Let $P_0,P_1,\dots,P_4$ be five points such that the four edges $P_0P_i$ for $i =1,2,3,4$ are tangent to $\unip$. Let $T_1, T_2, T_3, T_4$ be the corresponding unit edge vectors, $R_i$ the radii of the spheres at the vertices, so that 
\begin{align}
	P_{i} - P_0 = (R_{i} + R_0) T_i.
\end{align}
Let $K_i$ be the touching point on the edge $P_0P_{i}$, so that $K_i - P_0 = R_0 T_i$. 

As before, the dual points are defined by
\begin{align}
	P^*_{i} - P^*_0 = \pm (R_{i}^{-1} + R_0^{-1}) T_i = (R^*_{i} + R^*_0) T^*_i,
\end{align}
where the sign is $+$ for  horizontal edges and $-$ for vertical edges. The dual contact points satisfy $K^*_i - P_0^* = R^*_0T_i^*$. 

The dual points in the vertex star of the associated family are obtained by integrating
\begin{align}
	P^\varphi_{i} - P^\varphi_0 = \pm (R^*_{i} + R^*_0) T_i^\varphi \label{eq:Xassointegral}
\end{align}
with $T_i^\varphi$  determined by \eqref{eq:associatedtangent}. 

Note that the face normal $n$ in \eqref{eq:associatedtangent} is not the same for the unit edge vectors $T^\varphi_1, T^\varphi_2, T^\varphi_3, T^\varphi_4$ of the vertex star. One can derive formulas for $T^\varphi_i$, $N^\varphi_i$ and $B^\varphi_i$, analogous to \eqref{eq:associatedtangent}, \eqref{eq:associatedtangent1} and \eqref{eq:associatedtangent2}, for a vertex star of $\isow$. Instead of the face normal(s) $n$, one considers the vertex normal $n_0$ at $P_0$, given by the connection of the center of $\unip$ and $P_0$. Then one can express $T^\varphi_i$, $N^\varphi_i$ and $B^\varphi_i$ in the orthonormal basis $(t_i, n_0, B_i)$ where $t_i$ is given by $t_i := n_0 \times B_i$ so that
	\begin{align}
		T_i^\varphi &= t_i \cosh \theta \cos \varphi + n_0 \sinh \theta \cos \varphi + B_i \sin \varphi,\\
	N_i^\varphi &=  t_i \sinh \theta +  n_0 \cosh \theta, \\
		B_i^\varphi &= -t_i \cosh \theta \sin \varphi - n_0 \sinh \theta \sin \varphi + B_i \cos \varphi.
	\end{align}
Here $\theta$ is the angle that rotates $(T_i, N_i, B_i)$ onto $(t_i, n_0, B_i)$ and is given by $\tanh{\Theta} = R_0$, in particular it independent of $i$.

The contact points $K^\varphi_i$ of the surface in the associated family are given by
\begin{align}
	K^\varphi_i = \pm T_i^\varphi R_0^*,
\end{align}
where the sign depends on the type of edge.
The direction vectors of the two isotropic lines $G^\varphi_{i,\pm}$ at $K^\varphi_i$ are again given by $N^\varphi_i \pm B^\varphi_i$. 
The isotropic line $G^\varphi_{i,+}$ through $K^\varphi_i$ intersects the isotropic line $G^\varphi_{i+1,+}$ through $K^\varphi_{i+1}$ in a point $Y_i$. Indeed, they lie in two touching timelike spheres, the sphere of radius $R^*_0$ centered at $P_0^*$, and the sphere of radius $\sin \varphi$ associated to the common adjacent face (see Figure \protect{\ref{fig:asso}}, left), and therefore must intersect.

Let $\Pi_0$ be the tangent plane at $P_0^*$, that is the plane orthogonal to $n_0$ passing through $P_0^*$.

\begin{lemma}\label{lem:zdistance}
  Let $Z_i$ be the projection of the intersection point $Y_i$ of consecutive isotropic lines onto the plane $\Pi_0$. The distance of $Z_i$ to  $P_0^*$ is independent of $\varphi$. In particular, it is given by
	\begin{align*}
		|Z_i - P_0^*| = \frac{R^*_0}{\cos(\alpha/2)},
	\end{align*}
	where $\alpha$ is the angle between $t_i$ and $t_{i+1}$. 
\end{lemma}

\proof{
	The intersection of $\Pi_0$ with the sphere around $P_0^*$ is a circle of Radius $R_0^*$. An isotropic line contained in the sphere intersects this circle in a point $V$. Since the isotropic line lies in the tangent plane at $V$ it projects to the tangent of the circle at $V$. Obviously, the distance of $V$ to $P_0^*$ is $R_0^*$.
	Hence, the projections of the isotropic lines are tangent to a circle with center $P_0^*$ and radius $R_0^*$. The distance of their intersection point $Z_i$ only depends on the angle the projected tangents form which is given by $\alpha$.
	To show that the angle $\alpha$ is indeed also independent of $\varphi$, we consider the projections of the isotropic lines $G^\varphi_{i,\pm}$ to $\Pi_0$, which in the orthonormal basis $(t_i, n_0, B_i)$ have the direction vectors
	\begin{align}
		t_i(\sinh \theta \mp \cosh \theta \sin \varphi) \pm B_i \cos \varphi. \label{eq:projectedtangentdirection}
	\end{align}
Thus the angle $\alpha$ is simply the angle $\alpha = \angle(t_i,t_{i+1}) = \angle(B_i, B_{i+1})$.
	Consequently, the distance is given by $R_0^* / \cos(\alpha/2)$.
	\qed
}

We are ready to prove the final theorem.

\begin{theorem}\label{th:xasso}
	The $X$-variables in the associated family of null congruences $\dot\cc^\varphi$ of a maximal congruence $\cc$ do not depend on $\varphi$.
\end{theorem}
\begin{proof}
	We consider a family of null congruences  $\dot\cc^\varphi$ and their projection to  $\eucl$, see also Figure \protect{\ref{fig:icnets}} (right).
	Recall that due to Definition~\ref{def:xvar} we have that
	\begin{align}	
		X_\circ(w) = -\frac{(\cpb(b_1)-\cpmw(w))(\cpb(b_2)-\cpmw(w))}{(\cpb(b_3)-\cpmw(w))(\cpb(b_4)-\cpmw(w))}, \label{eq:xvardiezweite}
	\end{align}	
	where $\cpmw(w)$ is the projection of $\dot\cc_\circ^\varphi(w)$ to $\eucl$ and $\cpb(b_i)$ is the projection of $\dot\cc_\bullet^\varphi(b_i)$ to $\eucl$. Moreover, there is a scaling of $\lor$ centered at $\dot\cc_\circ^\varphi(w)$ that takes the four points $\dot\cc_\bullet^\varphi(b_i)$ to the respective contact points $K_i$ of $\ccw(w)$ with $\ccb(b_i)$. Hence, we may calculate $X_\circ(w)$ using the projections  $L^\varphi_i$ of the contact points $K^\varphi_i$ via	\begin{align}	
	X_\circ(w) = -\frac{(L^\varphi_1-\cpmw(w))(L^\varphi_2-\cpmw(w))}{(L^\varphi_3-\cpmw(w))(L^\varphi_4-\cpmw(w))}, 
\end{align}	
	because Equation~\eqref{eq:xvardiezweite} is manifestly invariant under scaling. 
	Furthermore, due to Lemma~\ref{lem:xcyclo} the $X$-variables are Lorentz invariant, thus we can also project to any other spacelike plane instead (of $\eucl \subset \lor$).
	With this in mind, we project the four contact points $K^\varphi_i$  to the plane $\Pi_0$ of Lemma~\ref{lem:zdistance} instead of $\eucl$. This yields exactly the points $Z_i$ of Lemma~\ref{lem:zdistance}. Thus, with $\odot\dot\cp^\varphi_\circ(w) = \cpmw^\varphi(w) = P_0^*$ we obtain the expression
	\begin{align}	
		X_\circ(w) = -\frac{(Z_1-P_0^*)(Z_2-P_0^*)}{(Z_3-P_0^*)(Z_4-P_0^*)}.
	\end{align}			
	 Finally, Lemma~\ref{lem:zdistance} shows that the distances of $Z_i$ to $\cpmw(w) = P_0^*$ are independent of $\varphi$, hence so are the $X$-variables.
\end{proof}

\begin{remark}
	In \cite{admpsiso} we defined the $X$-variables for a general contact congruence, and we showed that the $X$-variables are independent of Laguerre transformations. Since $\cc^\varphi$ is related to $\dot\cc^\varphi$ by a Laguerre transformation, the $X$-variables of $\cc^\varphi$ coincide with those of $\dot\cc^\varphi$ and therefore also do not depend on $\varphi$.
\end{remark}

\begin{remark}
	Theorem~\ref{th:xasso} implies that the $X$-variables for the incircular nets in the associated family of a maximal incircular net are independent of $\varphi$. Thus each maximal incircular net comes with a pair of 1-parameter families of incircular nets that have the same $X$-variables, and have the properties as discussed in Remark~\ref{rem:assocp}.
\end{remark}

\begin{remark}
	For isothermic incircular nets, we showed in \cite{admpsiso} that the $X$-variables are in a subvariety, which we called the \emph{isothermic subvariety}. Theorem~\ref{th:xasso} shows that the associated incircular nets also have $X$-variables in the isothermic subvariety, but they are not isothermic incircular nets. That said, the name \emph{isothermic subvariety} may still be justified, since each $\cc^\varphi$ is \emph{still} an isothermic surface. The associated congruences are still discretizations of isothermic surfaces, just not in curvature-line parametrization. Moreover, as shown in \cite{admpsiso} the isothermic subvariety is a characterization of a certain periodicity in the Miquel dynamics, see also Remark~\ref{rem:miquel}. 
	In fact, it follows from the properties of the incircular nets $\cp^\varphi$ in the associated family, that their center nets $\cpm^\varphi$ have the same periodicity behaviour with respect to Miquel dynamics. However, it is currently unclear how the radii behave under Miquel dynamics.
\end{remark}

\bibliographystyle{alpha}
\bibliography{references}

\end{document}